\documentclass[12pt,a4paper,reqno]{amsart}

\usepackage{amsfonts,amsthm,amsmath,amssymb,amscd,mathrsfs}
\usepackage{graphics}
\usepackage{indentfirst}
\usepackage{cite}
\usepackage{bm, enumerate,xcolor}
\usepackage[dvips]{epsfig}
\usepackage{latexsym}
\usepackage{float}
\usepackage{mathtools}
\usepackage[hidelinks]{hyperref} 
\usepackage{appendix}
\usepackage{accents}

\usepackage [english]{babel}
\usepackage [autostyle, english = american]{csquotes}
\MakeOuterQuote{"}

\newcommand{\ubar}[1]{\underaccent{\bar}{#1}}

\newtheorem{theorem}{Theorem}[section]
\newtheorem{remark}{Remark}[section]
\newtheorem{example}{Example}[section]
\newtheorem{definition}{Definition}[section]
\newtheorem{lemma}[theorem]{Lemma}
\newtheorem{pro}[theorem]{Proposition}
\newtheorem{coro}[theorem]{Corollary}

\def\be{\begin{eqnarray}}
	\def\ee{\end{eqnarray}}
\def\ba{\begin{aligned}}
	\def\ea{\end{aligned}}
\def\bay{\begin{array}}
	\def\eay{\end{array}}
\def\bca{\begin{cases}}
	\def\eca{\end{cases}}

\def\bt{\begin{theorem}}
	\def\et{\end{theorem}}
\def\bc{\begin{corollary}}
	\def\ec{\end{corollary}}
\def\bl{\begin{lemma}}
	\def\el{\end{lemma}}
\def\bp{\begin{proposition}}
	\def\ep{\end{proposition}}
\def\br{\begin{remark}}
	\def\er{\end{remark}}
\def\bd{\begin{definition}}
	\def\ed{\end{definition}}
\def\bpf{\begin{proof}}
	\def\epf{\end{proof}}
\def\bex{\begin{example}}
	\def\eex{\end{example}}
\def\bq{\begin{question}}
	\def\eq{\end{question}}
\def\bas{\begin{assumption}}
	\def\eas{\end{assumption}}
\def\ber{\begin{exercise}}
	\def\eer{\end{exercise}}
\def\mb{\mathbb}
\def\mbR{\mb{R}}
\def\mbT{\mb{T}}
\def\mbZ{\mb{Z}}
\def\mc{\mathcal}
\def\mcS{\mc{S}}

\def\mcH{\mathcal{H}}

\def\Div{{\rm div}}
\def\be{\mathbf{e}}
\def\bi{\mathbf{i}}

\def\bu{\mathbf{u}}
\def\bv{\mathbf{v}}
\def\bV{\mathbf{V}}
\def\bbf{\mathbf{f}}

\def\bn{\mathbf{n}}

\def\mfk{\mathfrak{k}}

\def\rone{{\rm\uppercase\expandafter{\romannumeral 1}}}
\def\rtwo{{\rm\uppercase\expandafter{\romannumeral 2}}}
\def\thr{{\rm\uppercase\expandafter{\romannumeral 3}}}

\DeclareMathOperator{\sgn}{sgn}

\def\XXint#1#2#3{{\setbox0=\hbox{$#1{#2#3}{\int}$ }
		\vcenter{\hbox{$#2#3$ }}\kern-.6\wd0}}


\begin{document}


\title[Steady solutions to Euler equations]{On a classification of steady solutions to two-dimensional Euler equations}

\author{Changfeng Gui}
\address{Department of Mathematics,  University  of  Macau,  Macau  SAR,  P. R. China}
\email{changfenggui@um.edu.mo}

\author{Chunjing Xie}
\address{School of Mathematical Sciences, Institute of Natural Sciences, Ministry of Education Key Laboratory of Scientific and Engineering Computing, CMA-Shanghai, Shanghai Jiao Tong University, Shanghai 200240, China.}
\email{cjxie@sjtu.edu.cn}

\author{Huan Xu}
\address{Department of Mathematics, The University of Texas at San Antonio, San Antonio, TX 78249, USA}
\email{huan.xu@utsa.edu}

\date{}

\maketitle

\begin{abstract}
In this paper, we provide a classification of steady solutions to two-dimensional incompressible Euler equations in terms of the set of flow angles. The first main result asserts that the set of flow angles of any bounded steady flow in the whole plane must be the whole circle unless the flow is a parallel shear flow. In an infinitely long horizontal strip or the upper half-plane supplemented with slip boundary conditions, besides the two types of flows appeared in the whole space case, there exists an additional class of steady flows for which the set of flow angles is either the upper or lower closed semicircles. This type of flows is proved to be the class of non-shear flows that have the least total curvature. As applications, we obtain Liouville-type theorems for two-dimensional semilinear elliptic equations with only bounded and measurable nonlinearity, and the structural stability of shear flows whose all stagnation points are not inflection points, including Poiseuille flow as a special case. Our proof relies on the analysis of some quantities related to the curvature of the streamlines.
\end{abstract}


\section{Introduction and main results}

In this paper, we study steady solutions to the Euler system in a domain $\Omega\subset\mbR^2$
\begin{eqnarray}\label{Euler}
\left\{\begin{aligned}
&\bv\cdot\nabla\bv+\nabla P=0, & {\rm in}\ \Omega,\\
&\Div\bv=0, & {\rm in}\ \Omega.
\end{aligned}\right.
\end{eqnarray}
Here $\bv=(v_1,v_2)$ is the velocity field and $P$ is the pressure. We are mainly concerned with three unbounded domains: the whole plane $\mbR^2$, the half-plane $\mbR^2_+\vcentcolon=\mbR\times(0,\infty)$, and the infinitely long strip $\Omega_\infty\vcentcolon=\mbR\times(-1,1)$. Since we shall be working in the class of continuous functions, the results obtained in this paper will also apply to steady flows on the torus $\mbT^2$, the $x_1$-periodic half-plane $\mbT\times(0,\infty)$, and the $x_1$-periodic strip $\mbT\times(-1,1)$, where $\mbT=\mbR/2\pi\mbZ$. If $\Omega$ is either $\Omega_\infty$ or $\mbR_+^2$ , the flows are supplemented with the slip boundary condition
\begin{align}\label{slip boundary condition}
v_2=0\ \ {\rm on}\ \partial\Omega.
\end{align}

The Euler system \eqref{Euler} has two important classes of solutions with symmetry: the parallel shear flows (i.e., the flows whose streamlines consist of parallel lines) and the circular flows (i.e., the flows whose streamlines consist of concentric circles). Then a natural question is whether the flows have certain rigidity when the domain enjoys some symmetric property.  For example, in the aforementioned domains, one investigates the conditions that lead to steady Euler flows being shear flows. The rigidity issue has attracted significant attention in mathematical research, since it has close relationship with stability and asymptotic properties of the solutions (see \cite{CDG, CEW,LZ}). 

In a series of works \cite{HN,HN2,HN3}, Hamel and Nadirashvili thoroughly studied the rigidity of steady Euler flows in various domains with symmetry. One of their main results asserts that steady flows strictly away from stagnation points are rigid. The key point in \cite{HN,HN3} is to show that the stream functions of the flows are governed by semilinear elliptic equation with {\it Lipschitz continuous} nonlinearity, thereby enabling the application of classical elliptic techniques (\hspace{1sp}\cite{BCN,BN,GNN}) to prove symmetry. On the other hand, many flows of physical importance, such as Couette flows, Poiseuille flows,  Kolmogrov flows, and the vortex patch, etc., have stagnation points (cf. \cite{DR, ElgindiHuang}). 
Furthermore, the stagnation points of the flows play important role in understanding the singularity and growth bounds of incompressible Euler flows, one may refer to \cite{HouLuo, KS, Denisov} and references therein. 
Although the stagnation points produce some technical difficulties, 
the strategy of using semilinear elliptic equations has also proven successful in comprehending certain local structures of steady Euler flows with stagnation points (see \cite{N}). Combining the methods of moving plane and calculus of variations,
the Liouville type theorem for Poiseuille flows in a strip and the existence of solution for Euler equations in general nozzles was established in \cite{LLSX}, where the regularity of boundary of stagnation region was also achieved with the aid of analysis for obstacle type free boundaries. One may refer to \cite{R,WZ} for some recent study on circular flows with stagnation points. 
It was shown in \cite{GPSY} that any compactly supported and nonnegative vorticity of the smooth steady Euler flow must be radially symmetric up to a translation. 

We should mention that in contrast to rigidity,  the Euler system \eqref{Euler} in a symmetric domain may have many solutions without obvious symmetry. Therefore, it is also important to construct nontrivial steady solutions to \eqref{Euler}. 
In this direction, one may refer to \cite{CL,CDG,CEW,FMM,LZ} and references therein.

In this paper, we try to classify the steady Euler flows with possible stagnation points in the plane, half plane, or strip from a geometric point view.  The flows are classified based on their set of flow angles. Specifically, we show that the set of flow angles of any bounded steady flow in $\mbR^2$ must be the whole circle unless the flow is a parallel shear flow. In $\Omega_\infty$ or $\mbR_+^2$, due to the presence of boundaries, there exists one more class of steady flows for which the set of flow angles is either the upper or lower closed semicircles.

\subsection{Main theorems}
For a vector field $\bu$ defined in a domain $\Omega$, the set of flow angles of $\bu$ in $\Omega$ is defined as
\begin{align*}
\Theta(\bu;\Omega)\vcentcolon=\left\{\frac{\bu(x)}{|\bu(x)|}\Big|\, x\in \Omega, |\bu(x)|>0\right\},
\end{align*}
which is a subset of the circle $S^1$. For $\theta_0\in\mbR$ and $\beta\in(0,\pi)$, denote
\[
\mcS_{\theta_0,\beta}=\left\{(\cos\theta,\sin\theta)\in S^1| -\pi+\beta+\theta_0\le\theta\le \pi-\beta+\theta_0\right\},
\]
which is a closed arc segment on $S^1$ with length $2\pi-2\beta$. Later on, for brevity of notation, let $S^1_+=\mcS_{\pi/2,\pi/2}$ and $S^1_-=\mcS_{-\pi/2,\pi/2}$ be the upper and lower closed semicircles, respectively. 

Our first main theorem can be stated as follows.

\begin{theorem}\label{rigidity in the plane}
Suppose that $\bv=(v_1,v_2)\in H^1_{loc}(\mbR^2)\cap C(\mbR^2)$ is a solution of \eqref{Euler}. Assume there exists a constant $C>0$ such that 
\begin{align}\label{growth at far fields}
\sup_{|x|\le R}|\bv(x)|\le C \sqrt{\log R}\ \ \text{for all}\ R>100.   
\end{align}
Then the following statements hold:

{\rm (\romannumeral1)} If there exist $\theta_0\in\mbR$ and $\beta\in(0,\pi)$ such that $\Theta(\bv;\mbR^2)\subset \mcS_{\theta_0,\beta}$, then $\bv$ is a shear flow.

{\rm (\romannumeral2)} If $\bv\in W^{2,1}_{loc}(\mbR^2)$, then either $\bv$ is a shear flow, or $\Theta(\bv;\mbR^2)= S^1$.
\end{theorem}

There are a few remarks for Theorem \ref{rigidity in the plane}.
\begin{remark}
The results in Theorem~\ref{rigidity in the plane} are no longer valid if $\bv$ grows too fast at far fields. For example, let
\[
\bv\vcentcolon=(-e^{x_1},x_2 e^{x_1})\ \ {\rm and}\ \ P\vcentcolon=-\frac12e^{2x_1}.
\]
Then $(\bv, P)$ is a smooth solution of $\eqref{Euler}$ in $\mbR^2$. Although it satisfies $v_1<0$ so that $\Theta(\bv, \mathbb{R}^2)\subset \mcS_{\pi, \pi/2}$, it is not a shear flow.
\end{remark}

\begin{remark}

As a consequence of Theorem~\ref{rigidity in the plane}, we can prove Liouville-type theorems for a class of semilinear elliptic equations with only {\it bounded and measurable} nonlinearity (see Subsection~\ref{Liouville SEE}). But requiring a $W^{2,1}$ Sobolev regularity for $\bv$ may result in the non-existence of solutions to these elliptic equations (see Example~\ref{discontinuous nonlinearity example}). So we also establish Theorem~\ref{rigidity in the plane} {\rm (\romannumeral1)} (and Theorem~\ref{rigidity in a strip} {\rm (\romannumeral1)}) in which less regularity of $\bv$ is required. 
    
\end{remark}


Theorem~\ref{rigidity in the plane} {\rm (\romannumeral2)} classifies steady flows in $\mbR^2$ into two categories based on the set of flow angles. When the physical boundary appears, the steady flows of the Euler system may have richer structures.  The following theorem provides a classification of steady solutions to the Euler system in the half plane and in a strip.

\begin{theorem}\label{rigidity in a strip}
Let $\Omega=\Omega_\infty$ or $\Omega=\mbR_+^2$. Suppose that $\bv=(v_1,v_2)\in H^1_{loc}(\Bar{\Omega})\cap C(\Bar{\Omega})$ is a solution of \eqref{Euler} satisfying \eqref{slip boundary condition}. Assume that there exists a constant $C>0$ such that
\begin{align}\label{growth at horizontal far fields}
\sup_{[-R,R]\times[-1,1]}|\bv|\le C \sqrt{R},\ \ for\ all\ R>100
\end{align}
in the case  $\Omega=\Omega_\infty$, or
\begin{align}\label{growth at halfplane far fields}
\sup_{\{x\in\mbR_+^2:|x|\le R\}}|\bv|\le C \sqrt{\log R},\ \ \text{for all}\ R>100  
\end{align}
in the case $\Omega=\mbR_+^2$.
Then the following statements hold.

{\rm (\romannumeral1)} If there exists a number $\beta\in(0,\pi)$ such that
\[
\Theta(\bv;\Bar{\Omega})\subset \mcS_{0,\beta},\ or\ \Theta(\bv;\Bar{\Omega})\subset \mcS_{\pi,\beta},
\]
then $\bv$ is a shear flow, that is, $\bv\equiv(v_1(x_2),0)$.

{\rm (\romannumeral2)} If additionally $\bv\in C^1(\Bar{\Omega})\cap W^{2,1}_{loc}(\Bar{\Omega})$, and $v_1|_{\partial\Omega}\in L^\infty(\partial\Omega)$, then exactly one of the following three cases holds:
\begin{enumerate}
\item[(a)] $\bv$ is a shear flow;

\item[(b)] $\Theta(\bv;\Bar{\Omega})=S^1$;

\item[(c)] either $\Theta(\bv;\Bar{\Omega})=S^1_+$ or $\Theta(\bv;\Bar{\Omega})=S^1_-$.
\end{enumerate}

{\rm (\romannumeral3)}
Assume additionally that $\bv\in C^1(\Bar{\Omega})\cap W^{2,1}_{loc}(\Bar{\Omega})$ is $x_1$-periodic, i.e. there exists a number $T>0$ such that
\[
\bv(x_1+T,x_2) =\bv(x_1,x_2) \quad \text{for all}\,\, (x_1,x_2)\in \Omega.
\]
Then it holds that either $\bv$ is a shear flow or $\Theta(\bv;\Omega)=S^1$.
\end{theorem}

\begin{remark}
The conclusion in Theorem~\ref{rigidity in a strip} {\rm (\romannumeral3)} asserts that the set of flow angles of any non-shear flow in the interior region is the whole circle, which is better than that $\Theta(\bv;\Bar{\Omega})=S^1$. 
\end{remark}

In the rest of the paper, we call the flows belonging to Case (c) in Theorem~\ref{rigidity in a strip} (ii) Type {\thr}. The existence of Type {\thr} flows is guaranteed by the following theorem.

\begin{theorem}\label{existence_thm}
Let $\Omega=\Omega_\infty$ or $\Omega=\mbR_+^2$. There exists a bounded Type {\thr} flow $\bv\in C^\infty(\Bar{\Omega})$, i.e., there exists a bounded solution $\bv\in C^\infty(\Bar{\Omega})$ to \eqref{Euler} and \eqref{slip boundary condition}, which satisfies either $\Theta(\bv;\Bar{\Omega})=S^1_+$ or $\Theta(\bv;\Bar{\Omega})=S^1_-$.   
\end{theorem}

\subsection{Some Consequences}\label{Liouville SEE}
We introduce two consequences of the classification results obtained in Theorems~\ref{rigidity in the plane} and \ref{rigidity in a strip}.

First, it is well known that if $u$ is a solution to a semilinear elliptic equation in two dimensions, say 
\begin{equation}\label{elleq1}
  \Delta u=F'(u),  
\end{equation}
then $(\bv,P)\vcentcolon=(\nabla^\perp u,F(u)-\frac{1}{2}|\nabla u|^2)$ is a solution to the Euler system \eqref{Euler}, where $\nabla^\perp \vcentcolon=(-\partial_{x_2},\partial_{x_1})$. Furthermore, $u$ is a one-dimensional solution if and only if $\bv$ is a shear flow. Therefore, the following Liouville-type theorems for semilinear elliptic equations follow directly from Theorem~\ref{rigidity in the plane}.

\begin{coro}
Let $F:\mbR\rightarrow\mbR$ be a locally Lipschitz continuous function. Suppose that $u\in H^{2}_{loc}(\mbR^2)\cap C^1(\mbR^2)$ solves \eqref{elleq1} in $\mbR^2$.
Assume there exists a constant $C>0$ such that 
\begin{align*}
\sup_{|x|\le R}|\nabla u(x)|\le C \sqrt{\log R}\quad {\rm for}\ {\rm all}\ R>100.   
\end{align*}
Then $u$ is a one-dimensional solution, that is, $u=u(\be\cdot x)$ for some unit vector $\be\in\mbR^2$, if one of the following is satisfied:

{\rm (\romannumeral1)} There exist $\theta_0\in\mbR$ and $\beta\in(0,\pi)$ such that $\Theta(\nabla u;\mbR^2)\subset \mcS_{\theta_0,\beta}$.

{\rm (\romannumeral2)} $u\in W^{3,1}_{loc}(\mbR^2)$ and $\Theta(\nabla u;\mbR^2)\subsetneqq S^1$.
\end{coro}

\begin{remark}
Note that $F'$ is only {\it bounded and measurable}, the elliptic equation should be understood in the sense that $\Delta u\nabla u=\nabla F(u)$ a.e. in $\mbR^2$.
\end{remark}

Similarly, the next result is an immediate consequence of Theorem~\ref{rigidity in a strip}.

\begin{coro}\label{Liouville in a strip}
Let $\Omega=\Omega_\infty$ or $\Omega=\mbR_+^2$, and $F:\mbR\rightarrow\mbR$ be a locally Lipschitz continuous function. Suppose that $u\in H^{2}_{loc}(\Bar{\Omega})\cap C^1(\Bar{\Omega})$ solves
\begin{equation*}
\left\{
    \begin{aligned}
&\Delta u=F'(u),\ \ in\  \Omega;\\
&\partial_{x_1}u=0,\ \ on\ \partial\Omega.       
    \end{aligned}
    \right.
\end{equation*}
Assume there exists a constant $C>0$ such that in the case $\Omega=\Omega_\infty$, $u$ satisfies 
\begin{align*}
\sup_{[-R,R]\times[-1,1]}|\nabla u|\le C \sqrt{R}\quad for\ all\ R>100, 
\end{align*}
 or in the case $\Omega=\mbR_+^2$, $u$ satisfies
\begin{align*}
\sup_{\{x\in\mbR_+^2||x|\le R\}}|\nabla u|\le C \sqrt{\log R}\quad \text{for all}\ R>100. 
\end{align*}
 Then the following statements hold.

{\rm (\romannumeral1)} If there exists a number $\beta\in(0,\pi)$ such that
\[
\Theta(\nabla u;\Bar{\Omega})\subset \mcS_{-\pi/2,\beta},\ or\ \Theta(\nabla u;\Bar{\Omega})\subset \mcS_{\pi/2,\beta},
\]
then $u$ depends on $x_2$ only, that is, $u=u(x_2)$.

{\rm (\romannumeral2)} If additionally $u\in C^2(\Bar{\Omega})\cap W^{3,1}_{loc}(\Bar{\Omega})$, and $\partial_{x_2}u|_{\partial\Omega}\in L^\infty(\partial\Omega)$, then exactly one of the following holds: a) $u$ depends on $x_2$ only; b) $\Theta(\nabla u;\Bar{\Omega})=S^1$; c) either $\Theta(\nabla u;\Bar{\Omega})=\mcS_{0,\pi/2}$ or $\Theta(\nabla u;\Bar{\Omega})=\mcS_{\pi,\pi/2}$.
\end{coro}

In Section \ref{existence section}, we will establish the existence of bounded and smooth solutions $u$ satisfying $\Theta(\nabla u;\Bar{\Omega})=\mcS_{0,\pi/2}$, from which Theorem~\ref{existence_thm} will subsequently follow.

\begin{remark}
If $F'$ is Lipschitz continuous, the Liouville-type theorem for semilinear elliptic equations \eqref{elleq1} in an $n$-dimensional slab $\mbR^{n-1}\times (0,1)$ has been proved in \cite[Theorem 1.6]{HN} under the assumption that $\partial_{x_n}u\ge0$ in the slab. For $\Omega=\Omega_\infty$, note that the assumption $\Theta(\nabla u;\Bar{\Omega})\subset \mcS_{\pi/2,\beta}$ is weaker than that $\partial_{x_2}u\ge0$ in $\Bar{\Omega}$. Moreover, in Corollary \ref{Liouville in a strip}, $F'$ is only bounded and measurable. Hence Corollary \ref{Liouville in a strip} {\rm (\romannumeral1)} can be regarded as a generalization of \cite[Theorem 1.6]{HN} when $n=2$.
\end{remark}

We provide the following example to illustrate  Corollary \ref{Liouville in a strip} {\rm (\romannumeral1)} in which less regularity of $u$ is required than in \cite[Theorem 1.6]{HN}.

\begin{example}\label{discontinuous nonlinearity example}
Let $\Omega=\Omega_\infty$. Suppose that $u\in H^{2}_{loc}(\Bar{\Omega})\cap C^1(\Bar{\Omega})$ solves
\begin{equation}\label{eqinexample}
    \left\{
    \begin{aligned}
& \Delta u=\sgn(u),\ \ in\  \Omega;\\  
&u(x_1,\pm1)=\pm\frac12, \ x_1\in\mbR.      
    \end{aligned}
    \right.
\end{equation}
Assume that there exists a constant $C>0$ such that 
\begin{align*}
\sup_{[-R,R]\times[-1,1]}|\nabla u|\le C \sqrt{R},\ for\ all\ R>100,   
\end{align*}
and that there exists a number $\beta\in(0,\pi)$ such that 
\begin{align}\label{direction assumption in example}
\Theta(\nabla u;\Bar{\Omega})\subset \mcS_{\pi/2,\beta}.  
\end{align}
It then follows from Corollary \ref{Liouville in a strip} {\rm (\romannumeral1)} that $u$ depends on $x_2$ only. This, together with the assumption \eqref{direction assumption in example}, further implies that $u'(x_2)\ge0$. Note that the elliptic equation in \eqref{eqinexample} implies that  $((u')^2-2|u|)'=0$. So $u'=\sqrt{2|u|+c}$ for some constant $c$. Since $u$ has zeros and $2|u|+c\ge0$, one has $c\ge0$. Finally, one can explicitly solve the ODE and find that $c=0$, and the only solution is $u=\frac12 x_2|x_2|$.
\end{example}

Second, we give a consequence of Theorem~\ref{rigidity in a strip} {\rm (\romannumeral2)} by showing the structural stability of any shear flow $(s(x_2),0)$ in $\Omega_\infty$ as long as all zeros of $s$ are not inflection points. In fact, if $s$ is strictly convex (which is the case for Poiseuille flow), we show a stronger result that any traveling wave solution near $(s(x_2),0)$ must be a shear flow.

\begin{coro}\label{stability of P flow}
Consider the shear flow $\bv_{sh}=(s(x_2),0)$ in $\Omega=\Omega_\infty$, where $s\in C^2([-1,1])$.

{\rm (\romannumeral1)} Assume $\min_{[-1,1]}s''>0$. Suppose that $\bV(t;x)=\bv(x_1-ct,x_2)$ is a travelling wave solution to the time-dependent Euler system
\begin{eqnarray*}
\left\{\begin{aligned}
&\partial_t\bV+\bV\cdot\nabla\bV+\nabla \Pi=0, & t\in\mbR, x\in\Omega,\\
&\Div\bV=0, & t\in\mbR, x\in\Omega,
\end{aligned}\right.
\end{eqnarray*}
where $c\in\mbR$ and the profile $\bv=(v_1,v_2)\in C^2(\Bar{\Omega})\cap L^\infty(\Omega)$ satisfies the slip boundary condition \eqref{slip boundary condition}.
If 
\begin{equation}\label{smallinCor}
   \|\partial_{x_2}(\omega-\omega_{sh})\|_{L^\infty(\Omega)}<\min_{[-1,1]}s'', 
\end{equation} where $\omega:=\partial_{x_1}v_2-\partial_{x_2}v_1$ and $\omega_{sh}=-s'$ are the vorticities of $\bv$ and $\bv_{sh}$, respectively, then $\bv$ must be a shear flow.

{\rm (\romannumeral2)} If $s$ has finitely many zeros in $[-1,1]$, and $s''(c)\neq0$ whenever $s(c)=0$, then the shear flow $\bv_{sh}$ is structurally stable in $\Omega$ in the following sense: there exists a number $\varepsilon_0>0$ such that, for any solution $\bv=(v_1,v_2)\in C^2(\Bar{\Omega})$ of \eqref{Euler} and \eqref{slip boundary condition}, if 
\begin{equation}\label{smallincor_2}
    \|v_1-s\|_{W^{2,\infty}(\Omega)}\le\varepsilon_0,
\end{equation} then $\bv$ must be a shear flow.
\end{coro}

\begin{remark}
It was proved in \cite[Proposition 4.1]{CEW} that the  specific Poiseuille flow $(x_2^2,0)$ in the $x_1$-periodic strip $\mbT\times(-1,1)$ is stable.
The method developed in \cite{CEW} works under the assumption that the Euler flow is $x_1$-periodic and relies on the specific expression of the background Poiseuille flow.
Corollary  \ref{stability of P flow} gives the stability of a large class of shear flows in an infinitely long strip. Hence Corollary \ref{stability of P flow}  can be regarded as an improvement of \cite[Proposition 4.1]{CEW}  for the flows with a general form. 
\end{remark}

\begin{remark}
In general, shear flows whose stagnation points happen to be inflection points can be unstable. Couette flow and Kolmogorov flow are two famous shear flows satisfying this property (see \cite{CEW, LZ}). The stationary structure near generic unstable shear flows will be investigated in future work. 
\end{remark}

\subsection{Key ideas of the proofs for the main results}
Here we give a brief overview of our approach to proving both Theorems \ref{rigidity in the plane} and  \ref{rigidity in a strip}.
First, it follows  from \eqref{Euler} that one has divergence-form equation
\begin{equation}\label{divergence-form Euler}
\Div(v_1\nabla v_2-v_2\nabla v_1)=0\  \ {\rm in}\ \Omega.    
\end{equation}
Our strategy is to exploit the geometric property of equation \eqref{divergence-form Euler}. We start by assuming $\Theta(\bv;\Bar{\Omega})\subsetneqq S^1$, which makes the flow angle single-valued. Despite the presence of any stagnation set, one can regularize the flow angle which is then shown to satisfy a divergence-form elliptic equation due to \eqref{divergence-form Euler}. Then we carefully analyze the {\it total curvature} for the flow, which is defined as
\[
\int_\Omega(|\nabla\bv|^2-|\nabla|\bv||^2)\,dx.
\]
The name of the above quantity is attributed to the fact that
\begin{align}\label{square norm of second fundamental form}
\frac{|\nabla\bv|^2-|\nabla|\bv||^2}{|\bv|^2}=\kappa^2+\tau^2, 
\end{align}
where $\kappa$ is the curvature of the streamline and $\tau$ is the curvature of the curve perpendicular to the streamline. The total curvature can be expressed in various forms by utilizing the following identities
\begin{align*}
|\nabla\bv|^2-|\nabla|\bv||^2=|\bv|^2\left|\nabla\left(\frac{\bv}{|\bv|}\right)\right|^2=\frac{|v_1\nabla v_2-v_2\nabla v_1|^2}{|\bv|^2}=\frac{|\nabla P|^2}{|\bv|^2},
\end{align*}
where the last identity is due to the fact that $\nabla^\perp P=v_1\nabla v_2-v_2\nabla v_1$. The curvature estimate (which is the estimate of the quantity \eqref{square norm of second fundamental form} in our context) has been an important tool in the study of minimal surfaces since the discovery of Simons inequality, which was proved by Simons in his fundamental work \cite{Simons} on the resolution of the Bernstein problem. See also \cite{Schoen, SSY} for more related works, and \cite{CLi,WW} and the references therein for recent developments.

Now the proof of Theorem~\ref{rigidity in the plane} is reduced to demonstrating the triviality of the total curvature under the assumption that $\Theta(\bv;\mbR^2)\subsetneqq S^1$. However, in the presence of boundaries, the assumption that $\Theta(\bv;\Omega)\subsetneqq S^1$ can only result in the finiteness of the total curvature, namely,
\begin{align}\label{finite total curvature}
\int_\Omega\frac{|v_1\nabla v_2-v_2\nabla v_1|^2}{|\bv|^2}\,dx<\infty.    
\end{align}
This is sufficient for us to conclude Part {\rm (\romannumeral3)} of Theorem~\ref{rigidity in a strip}, but not Part {\rm (\romannumeral2)}. To prove Part {\rm (\romannumeral2)}, we introduce a slightly stronger assumption that $\Theta(\bv;\Bar{\Omega})\subsetneqq S^1$, which then leads to an identity relating the total curvature to a boundary term. On the other hand, for any steady Euler flows with finite total curvature in $\Omega_\infty$ or $\mbR^2_+$,  the total curvature is bounded from below by the boundary term (see Proposition \ref{theorem of best constant}). So this lower bound estimate is sharp and its extremal functions are those steady flows satisfying $\Theta(\bv;\Bar{\Omega})\subsetneqq S^1$. With the help of some rigidity lemmas, we further characterize the non-shear extremal functions as Type {\thr} flows. Thus, roughly speaking, the class of Type {\thr} flows consists of all non-shear flows that have the least total curvature.

To prove Theorem~\ref{existence_thm}, we construct a Type {\thr} flow in $\mbR_+^2$ using a saddle solution to the Allen-Cahn equation (see \cite{DFP}) as its stream function. Similarly, for $\Omega=\Omega_\infty$, we construct the stream function by finding  a positive solution to some semilinear elliptic equation in the half-infinite strip $(0,\infty)\times(-1,1)$. The solution in the whole strip is obtained by an odd reflection with respect to $x_1$.

Lastly, we point out that the symmetry of the set of flow angles for any steady Euler flow, as revealed by Theorems \ref{rigidity in the plane} and \ref{rigidity in a strip}, can be quantified. Roughly speaking, the total curvature of any steady Euler flow, if finite, is equally distributed along almost every direction. For the exact statement, please refer to Proposition~\ref{total curvature evenly distributed}. 

The rest of the paper is organized as follows. In the next section, we define two test functions that will be frequently used in this paper and then collect some elementary lemmas. Section \ref{section rigidity in whole plane} is devoted to the proof of Theorem~\ref{rigidity in the plane}. In Section \ref{section steady flows with finite total curvature}, we establish a sharp lower bound estimate of the total curvature for steady Euler flows in $\Omega_\infty$ and $\mbR_+^2$. In Section \ref{section rigidity in strip}, with the aid of the results obtained in Section \ref{section steady flows with finite total curvature}, we give the proof of Theorem~\ref{rigidity in a strip} and Corollary \ref{stability of P flow}, and prove a quantitative version of symmetry for the flow angles. Finally, we provide the existence of bounded Type {\thr} steady Euler flows in Section \ref{existence section}.

\section{Preliminaries}\label{preliminary section}

Throughout the paper, we frequently use two standard test functions defined as follows. Given any $R>100$, define
\begin{equation}\label{log cutoff test function}
\phi_{log,R}(x)=\left\{\begin{aligned}
&1, && {\rm if}\ |x|\le R,\\
&2-\frac{\log|x|}{\log R}, && {\rm if} \ R<|x|<R^2,\\
&0, && {\rm if}\ |x|\ge R^2,
\end{aligned}\right.
\end{equation}
where $x=(x_1,x_2)$, and 
\begin{equation}\label{cutoff function for strip}
\phi_{1,R}(x_1)=\left\{\begin{aligned}
&1, && {\rm if}\ |x_1|\le R,\\
&\frac{x_1}{R}+2, && {\rm if} \ -2R<x_1<-R,\\
&2-\frac{x_1}{R}, && {\rm if} \ R<x_1<2R,\\
&0, && {\rm if}\ |x_1|\ge 2R.
\end{aligned}\right.
\end{equation}

For any vector $\bu=(u_1,u_2)$, we denote $\bu^\perp=(-u_2,u_1)$. The following lemma is well-known for smooth solutions of \eqref{Euler}, and it can be justified directly for solutions in $H^1$.

\begin{lemma}
Given any domain $\Omega\subset\mbR^2$. Let $\bv=(v_1,v_2)\in H_{loc}^1(\Omega)$ and $P\in \mathscr{D}'(\Omega)$ (the set of distributions in $\Omega$) be a solution of \eqref{Euler}. Then $\bv$ is a distributional solution of \eqref{divergence-form Euler}.
\end{lemma}
\begin{proof}
Since $(\bv,P)$ solves \eqref{Euler}, for any $\phi\in C_0^\infty(\Omega)$, we have
\[
-\int_\Omega (\bv\cdot\nabla\bv^\perp)\cdot\nabla\phi\,dx=\int_\Omega (\bv\cdot\nabla\bv)\cdot\nabla^\perp\phi\,dx=0.
\]
Then the lemma follows from the fact that $-\bv\cdot\nabla\bv^\perp=v_1\nabla v_2-v_2\nabla v_1$.
\end{proof}

The next lemma says that any incompressible flow satisfying the slip boundary condition \eqref{slip boundary condition} and having trivial total curvature must necessarily be a shear flow.

\begin{lemma}\label{trivial curvature lemma}
Given $-\infty\le a<b\le\infty$. Let $\Omega=\mbR\times(a,b)$. Suppose that $\bv=(v_1,v_2)\in H_{loc}^1(\Bar{\Omega})\cap C(\Bar{\Omega})$ solves
\begin{eqnarray*}
\left\{\begin{aligned}
&v_1\nabla v_2-v_2\nabla v_1=0, &\ \ {\rm in}\ \Omega,\\
&\Div\bv=0, &\ \ {\rm in}\ \Omega.
\end{aligned}\right.
\end{eqnarray*}
If either $a$ or $b$ is finite, we also assume the slip boundary condition \eqref{slip boundary condition} on the boundary $\partial \Omega$.
Then $\bv$ is a shear flow.
\end{lemma}
\begin{proof}
Assume $v_2$ is not identical to zero. We have $\nabla(\frac{v_1}{v_2})=0$ in $\{v_2\neq0\}$. So for every connected component of $\{v_2\neq0\}$, there exists a constant $c$ such that $v_1=cv_2$ in that component. This, together with the incompressibility, further implies that $((c,1)\cdot\nabla)\bv=0$. This means that $\bv$ is not only parallel to the direction $(c,1)$ but also constant along this direction. 
Hence one has $\bv= (cv_2(x_1-cx_2), v_2(x_1-cx_2))$. This implies that $\bv\equiv 0$ or $\bv= (cv_2(x_1-cx_2), v_2(x_1-cx_2))$ if $\Omega=\mbR^2$, and that $\bv=(v_1(x_2),0)$ if either $a$ or $b$ is finite when the slip boundary condition $v_2=0$ is assigned on the boundary.
Thus the proof of the lemma is completed.
\end{proof}


\section{Steady flows in the whole plane}\label{section rigidity in whole plane}

This section is devoted to the proof of Theorem \ref{rigidity in the plane}.

It is known that \eqref{divergence-form Euler} can be reformulated into a divergence-form elliptic equation under the assumption that $v_1>0$ (\hspace{1sp}\cite{GG, AC1}) or $\bv$ is strictly away from stagnation points (\hspace{1sp}\cite{HN2}). Now the key obstruction is that the flow may have stagnation points. Fortunately, under the assumption that $\Theta(\bv;\mbR^2)\subsetneqq S^1$, we can still derive a divergence-form elliptic equation satisfied by the regularized flow angle.

Given a unit vector $\be$ and any vector $\bu$ that does not point towards the direction of $-\be$, the angle from $\be$ to $\bu$, denoted by $\angle(\be,\bu)$, is defined as
\begin{eqnarray*}
\angle(\be,\bu)=\left\{\begin{aligned}
&\sgn(\bu\cdot\be^\perp)\cos^{-1}\left(\frac{\bu\cdot\be}{|\bu|}\right), && {\rm if}\ \bu\cdot\be^\perp\neq 0,\\
&0, && {\rm if}\ \bu=|\bu|\be.
\end{aligned}\right.
\end{eqnarray*}
Note that the angle from $\be$ to the zero vector is defined as $0$.

Throughout, we denote $\bi$ the unit vector $(1,0)$. 
Let us now give the proof of Theorem \ref{rigidity in the plane}.  Our proof is partially inspired by the study on De Giorgi conjecture in \cite{GG,AC1}.

\begin{proof}[Proof of Theorem \ref{rigidity in the plane}]
Without loss of generality, we assume $-\bi\notin \Theta(\bv;\mbR^2)$ and denote $\theta=\angle(\bi, \bv)$. Given any $\varepsilon>0$, let $\theta_\varepsilon=\angle(\bi,\bv_\varepsilon)$ be the angle from $\bi$ to $\bv_\varepsilon\vcentcolon=(v_1+\varepsilon,v_2)$. Note that $\bv_\varepsilon(x)\notin(-\infty,\varepsilon)\times\{0\}$ for all $x\in\mbR^2$. Hence $\theta_\varepsilon\in H^1_{loc}(\mbR^2)\cap C(\mbR^2)$ and the straightforward computations yeild
\begin{align}\label{gradient of theta}
\nabla\theta_\varepsilon=\frac{(v_1+\varepsilon)\nabla v_2-v_2\nabla v_1}{|\bv_\varepsilon|^2}.  
\end{align}
In view of \eqref{divergence-form Euler}, one has
\begin{align}\label{div form ee}
\Div\left(|\bv_\varepsilon|^2\nabla\theta_\varepsilon\right)=\varepsilon\Delta v_2 \quad \text{in}\,\, \mathscr{D}'(\mathbb{R}^2).
\end{align}

In the rest of the proof, we choose $\phi=\phi_{log, R}$. Taking  $\phi^2\theta_\varepsilon \in H_0^1(B_{R^2})$ as a test function for \eqref{div form ee} yields
\begin{align}\label{energy identity}
\int_{\mbR^2}\phi^2|\bv_\varepsilon|^2\left|\nabla\theta_\varepsilon\right|^2\,dx
=&-2\int_{\mbR^2}\phi\theta_\varepsilon|\bv_\varepsilon|^2\nabla\phi\cdot\nabla\theta_\varepsilon\,dx
-\varepsilon\int_{\mbR^2}\phi^2\theta_\varepsilon\Delta v_2\,dx.
\end{align}

Let us prove Part {\rm (\romannumeral1)} first. We may assume $\Theta(\bv;\mbR^2)\subset \mcS_{0,\beta}$ since the Euler system in $\mbR^2$ is invariant under rotation.
Note that $\theta = \angle(\bi, \bv)\in (-\pi+\beta, \pi-\beta)$. If $v_1\in (-\infty, -\frac{3}{2}\varepsilon)\cup (-\frac{\varepsilon}{2}, \infty)$, then one has $|\bv_\varepsilon|\geq \frac{\varepsilon}{2}$. If $v_1\in (-\frac{3}{2}\varepsilon, -\frac{\varepsilon}{2})$, then 
\[
|\bv_\varepsilon|\geq |v_2|=|v_1|\cdot|\tan\theta| \geq |v_1|\cdot|\tan\beta|\geq \frac{|\tan\beta|}{2}\varepsilon.
\]
Hence there is a positive constant $c$ depending only on $\beta$ such that
\begin{align}\label{lower bound of v}
|\bv_\varepsilon|\ge c\varepsilon.    
\end{align}
Integrating by parts gives
\begin{align}\label{extra term}
-\varepsilon\int_{\mbR^2}\phi^2\theta_\varepsilon\Delta v_2\,dx=\varepsilon\int_{\mbR^2}\phi^2\nabla\theta_\varepsilon\cdot\nabla v_2\,dx+2\varepsilon\int_{\mbR^2}\phi\theta_\varepsilon\nabla\phi\cdot\nabla v_2\,dx.
\end{align}
Since $|\theta_\varepsilon|\le\pi$, the dominated convergence theorem yields that
\begin{equation}\label{est1}
\lim_{\varepsilon\to 0}2\varepsilon\int_{\mbR^2}\phi\theta_\varepsilon\nabla\phi\cdot\nabla v_2\,dx =0.   
\end{equation}
As to the first integral on the right of \eqref{extra term}, combining \eqref{gradient of theta} and \eqref{lower bound of v} implies that
\[
|\varepsilon\phi^2\nabla\theta_\varepsilon\cdot\nabla v_2|\le \frac{C\varepsilon\phi^2|\nabla\bv|^2}{|\bv_\varepsilon|}
\le C\phi^2|\nabla\bv|^2\in L^1(\mbR^2).
\]
Also, $\frac{\varepsilon}{|\bv_\varepsilon|}$ converges to $0$ pointwisely in $\{v_2\neq0\}$ as $\varepsilon\rightarrow0$. So applying the dominated convergence theorem yields
\begin{equation}\label{est2}
    \begin{aligned} 
\lim_{\varepsilon\rightarrow0}\varepsilon\int_{\mbR^2}\phi^2\nabla\theta_\varepsilon\cdot\nabla v_2\,dx=& \lim_{\varepsilon\rightarrow0}\varepsilon\int_{\{v_2\neq0\}}\phi^2\nabla\theta_\varepsilon\cdot\nabla v_2\,dx + \lim_{\varepsilon\rightarrow0}\varepsilon\int_{\{v_2\equiv 0\}}\phi^2\nabla\theta_\varepsilon\cdot\nabla v_2\,dx \\   
= &\lim_{\varepsilon\rightarrow0}\varepsilon\int_{\{v_2\neq0\}}\phi^2\nabla\theta_\varepsilon\cdot\nabla v_2\,dx\\
=& 0.
    \end{aligned}
\end{equation}
Combining \eqref{est1} and \eqref{est2}  gives
\begin{align}\label{2nd integral tends to zero}
\lim_{\varepsilon\rightarrow0}\varepsilon\int_{\mbR^2}\phi^2\theta_\varepsilon\Delta v_2\,dx=0.
\end{align}

Now passing to the limit $\varepsilon\rightarrow0$ in \eqref{energy identity} and recalling $\theta =\angle(\bi, \bv)$ yield
\begin{equation}\label{limit energy identity}
\begin{aligned}
&\int_{\{v_2\neq0\}}\frac{\phi^2|v_1\nabla v_2-v_2\nabla v_1|^2}{|\bv|^2}\,dx\\
=& -2\int_{\{v_2\neq0\}}\phi\theta\nabla\phi\cdot(v_1\nabla v_2-v_2\nabla v_1)\,dx\\
\le&2\pi\left(\int_{\{v_2\neq0\}\cap\{|x|>R\}}\frac{\phi^2|v_1\nabla v_2-v_2\nabla v_1|^2}{|\bv|^2}\,dx\right)^{\frac12} \left(\int_{\mbR^2}|\bv|^2|\nabla\phi|^2\,dx\right)^{\frac12},
\end{aligned}  
\end{equation}
where the property that the gradient of $\phi=\phi_{log, R}$ has support on $\{|x|>R\}$ has been used.
By the definition of $\phi_{log,R}$ and the assumption \eqref{growth at far fields}, one has
\begin{align*}
\int_{\mbR^2}|\bv|^2|\nabla\phi|^2\,dx
\le\int_{\{R<|x|<R^2\}}\frac{C}{|x|^2\log R}\,dx\le C.
\end{align*}
As a result, we acquire 
\begin{align}\label{AC trick}
\int_{\{v_2\neq0\}}\frac{\phi^2|v_1\nabla v_2-v_2\nabla v_1|^2}{|\bv|^2}\,dx
\le C\left(\int_{\{v_2\neq0\}\cap\{|x|>R\}}\frac{\phi^2|v_1\nabla v_2-v_2\nabla v_1|^2}{|\bv|^2}\,dx\right)^{\frac12}.
\end{align}
It follows immediately that
\begin{align*}
\int_{\{v_2\neq0\}}\frac{\phi^2|v_1\nabla v_2-v_2\nabla v_1|^2}{|\bv|^2}\,dx\le C.
\end{align*}
Since the constant $C$ is independent of $R$, we have
\begin{align}\label{finite total curvature in the plane}
\int_{\{v_2\neq0\}}\frac{|v_1\nabla v_2-v_2\nabla v_1|^2}{|\bv|^2}\,dx\le C.
\end{align}

Based on \eqref{finite total curvature in the plane}, now letting $R\rightarrow\infty$ in \eqref{AC trick}, we further infer that the total curvature must be $0$.
Consequently, we have $v_1\nabla v_2-v_2\nabla v_1=0$ a.e. in $\{v_2\neq0\}$. On the other hand, it is obvious   that $v_1\nabla v_2-v_2\nabla v_1=0$ a.e. in $\{v_2 = 0\}$. Finally, recalling Lemma \ref{trivial curvature lemma} finishes the proof of Part {\rm (\romannumeral1)}.

To prove Part {\rm (\romannumeral2)}, we assume $\Theta(\bv,\mbR^2)\subsetneqq S^1$ and need to prove that $\bv$ is a shear flow. Again, we may assume $-\bi\notin \Theta(\bv,\mbR^2)$. Now since $\bv\in W^{2,1}_{loc}(\mbR^2)$, \eqref{2nd integral tends to zero} is evident. So the previous proof also allows us to conclude that $\bv$ is a shear flow. 

The proof of Theorem \ref{rigidity in the plane} is completed.
\end{proof}


\section{Flows in half-plane or in strips}\label{section steady flows with finite total curvature}

In this section, we denote $\phi_R=\phi_{log,R}$ defined in \eqref{log cutoff test function} if $\Omega=\mbR^2$ or $\Omega=\mbR^2_+$, and $\phi_R=\phi_{1,R}$ defined in \eqref{cutoff function for strip} if $\Omega=\Omega_\infty$.
First, inspired by the proof of Theorem \ref{rigidity in the plane}, we have the following two propositions on the rigidity of flows in half-plane or in strips.

\begin{pro}
\label{rigidity lemma when missing one direction}
Let $\Omega=\Omega_\infty$ or $\Omega=\mbR_+^2$. Suppose that $\bv=(v_1,v_2)\in H^1_{loc}(\Bar{\Omega})\cap C(\Bar{\Omega})\cap W^{2,1}_{loc}(\Bar{\Omega})$ is a solution of the steady Euler system  \eqref{Euler} satisfying the slip boundary condition \eqref{slip boundary condition}. Assume that $\bv$ satisfies \eqref{growth at horizontal far fields} and \eqref{growth at halfplane far fields} in the case $\Omega=\Omega_\infty$ and $\Omega=\mbR_+^2$, respectively. If
\begin{align*}
\Theta(\bv;\Bar{\Omega})\subset S^1\setminus\{-\bi\}\,\,\,\, or\,\,\,\, \Theta(\bv;\Bar{\Omega})\subset S^1\setminus\{\bi\},
\end{align*}
then $\bv$ is a shear flow.   
\end{pro}
\begin{proof}
The proof is almost the same as that of \autoref{rigidity in the plane}. We only point out that, assuming $-\bi\notin \Theta(\bv;\Bar{\Omega})$ and denoting $\theta_\varepsilon$ the angle from $\bi$ to $(v_1+\varepsilon,v_2)$, one should test \eqref{div form ee} by $\phi_{R}^2\theta_\varepsilon$ and note that $\theta_\varepsilon$ vanishes on the boundary $\partial\Omega$. The rest of the proof is omitted.
\end{proof}

\begin{pro}\label{rigidity lemma when missing two directions}
Let $\Omega=\Omega_\infty$ or $\Omega=\mbR_+^2$. Suppose that $\bv=(v_1,v_2)\in C^1(\Bar{\Omega})\cap W^{2,1}_{loc}(\Bar{\Omega})$ is a solution of the steady Euler system \eqref{Euler} satisfying the slip boundary condition \eqref{slip boundary condition}. Assume that $\bv$ satisfies \eqref{growth at horizontal far fields} and \eqref{growth at halfplane far fields} in the case $\Omega=\Omega_\infty$ and $\Omega=\mbR_+^2$, respectively. If there exist two unit vectors $\be=(e_1,e_2)$ and $\bbf=(f_1,f_2)$ with $e_2f_2<0$, such that
\begin{align*}
\Theta(\bv;\Omega)\subset S^1\setminus\{-\be,-\bbf\},
\end{align*}
then $\bv$ is a shear flow.
\end{pro}
\begin{proof}
The proof is divided into three steps.

{\bf Step 1.} We shall derive an energy identity analogous to \eqref{limit energy identity}. However, the identity includes a boundary term since the missing directions are slant. 

Note that $-\be\notin \Theta(\bv;\Omega)$ and denote
$\bv_\varepsilon\vcentcolon=(\bv\cdot\be+\varepsilon)\be+(\bv\cdot\be^\perp)\be^\perp$. Similar to the initial step in proving Theorem \ref{rigidity in the plane}, we define $\theta_\varepsilon=\angle(\be,\bv_\varepsilon)$ as the angle from $\be$ to $\bv_\varepsilon$. The straightforward computations show that $\theta_\varepsilon$ satisfies 
\begin{equation}\label{gradthetaeps}
\nabla\theta_\varepsilon=\frac{(\bv\cdot\be+\varepsilon)\nabla (\bv\cdot\be^\perp)-(\bv\cdot\be^\perp)\nabla(\bv\cdot\be)}{|\bv_\varepsilon|^2}  
\end{equation}
and 
\begin{equation}\label{eqthetaeps}
\Div\left(|\bv_\varepsilon|^2\nabla\theta_\varepsilon\right)=\varepsilon\Delta (\bv\cdot\be^\perp).
\end{equation}
Testing the equation \eqref{eqthetaeps} by $\phi\theta_\varepsilon$ with $\phi\in C_c^1(\Bar{\Omega})$, integrating by parts, and using \eqref{gradthetaeps} yield
\begin{align*}
&\int_{\Omega}\frac{\phi|(\bv\cdot\be+\varepsilon)\nabla (\bv\cdot\be^\perp)-(\bv\cdot\be^\perp)\nabla(\bv\cdot\be)|^2}{|\bv_\varepsilon|^2}\,dx\\
=&-\int_{\Omega}\theta_\varepsilon\nabla\phi\cdot\left[(\bv\cdot\be+\varepsilon)\nabla (\bv\cdot\be^\perp)-(\bv\cdot\be^\perp)\nabla(\bv\cdot\be)\right]\,dx\\
&+\int_{\partial\Omega}\phi\theta_\varepsilon\left[(\bv\cdot\be+\varepsilon)\partial_{\bn}(\bv\cdot\be^\perp)-(\bv\cdot\be^\perp)\partial_{\bn}(\bv\cdot\be)\right]\,dx_1
-\int_{\Omega}\varepsilon\phi\theta_{\varepsilon}\Delta(\bv\cdot\be^\perp)\,dx.
\end{align*}
Since $\be$ is slant and $\partial\Omega$ is horizontal, we see that 
\[
\lim_{\varepsilon\rightarrow 0}\theta_\varepsilon(x)=\theta(x)\vcentcolon=\angle(\be,\bv(x))\quad \text{for\ every}\,\, x\in\partial\Omega\ \text{with}\ v_1(x)\neq0.
\]
So applying the dominated convergence theorem in the process of $\varepsilon\rightarrow0$ and noticing that
\[
(\bv\cdot\be)\nabla (\bv\cdot\be^\perp)-(\bv\cdot\be^\perp)\nabla(\bv\cdot\be)=v_1\nabla v_2-v_2\nabla v_1,
\]
we arrive at
\begin{align}\label{energy identity in a strip}
&\int_{\Omega}\frac{\phi|v_1\nabla v_2-v_2\nabla v_1|^2}{|\bv|^2}\,dx
=-\int_{\Omega}\theta\nabla\phi\cdot(v_1\nabla v_2-v_2\nabla v_1)\,dx
+\int_{\partial\Omega}\phi\theta v_1\partial_{\bn}v_2\,dx_1.
\end{align}

{\bf Step 2.} We need to handle the boundary term in \eqref{energy identity in a strip}.
Let $\Bar{\theta}=\angle(\be,\bi)$ be the angle from $\be$ to $\bi$. If $\be$ points upward, that is, $e_2>0$, then \[
\theta=\Bar{\theta}\,\,\text{on}\,\, \{x\in\partial\Omega| v_1>0\} \quad  \text{and}\quad \theta=\Bar{\theta}+\pi\,\, \text{on}\,\, \{x\in\partial\Omega| v_1<0\}.
\]
Namely, $\theta=\Bar{\theta}+\frac{\pi}{2}-\frac{\pi}{2}\sgn(v_1)$ on $\{x\in\partial\Omega| v_1\neq0\}$ in the case of $e_2>0$. Similarly, if $\be$ points downward, then $\theta=\Bar{\theta}-\frac{\pi}{2}+\frac{\pi}{2}\sgn(v_1)$ on $\{x\in\partial\Omega| v_1\neq0\}$. In summary, we have 
\[
\theta v_1=\left(\Bar{\theta}+\frac{\pi}{2}\sgn(e_2)\right)v_1-\frac{\pi}{2}\sgn(e_2)|v_1|,\ \ on\ \partial\Omega.
\]
Hence we have
\begin{align}\label{handle boundary integral}
\int_{\partial\Omega}\phi\theta v_1\partial_{\bn}v_2\,dx_1=\left(\Bar{\theta}+\frac{\pi}{2}\sgn(e_2)\right)\int_{\partial\Omega}\phi v_1\partial_{\bn}v_2\,d x_1-\frac{\pi}{2}\sgn(e_2)\int_{\partial\Omega}\phi |v_1|\partial_{\bn}v_2\,d x_1.
\end{align}

Next, we test \eqref{divergence-form Euler} directly by $\phi$ to get
\begin{align*}
\int_{\partial\Omega}\phi v_1\partial_{\bn}v_2\,dx_1=\int_{\Omega}\nabla\phi\cdot(v_1\nabla v_2-v_2\nabla v_1)\,dx.
\end{align*}
This along with \eqref{energy identity in a strip} and \eqref{handle boundary integral} implies that for all $\phi\in C_0^1(\Bar{\Omega})$, one has
\begin{align}\label{energy identity when missing one direction}
\int_{\Omega}\frac{\phi|v_1\nabla v_2-v_2\nabla v_1|^2}{|\bv|^2}\,dx
=&\int_{\Omega}\left(\angle(\be,\bi)+\frac{\pi}{2}\sgn(e_2)-\angle(\be,\bv)\right)\nabla\phi\cdot(v_1\nabla v_2-v_2\nabla v_1)\,dx\nonumber\\
&-\frac{\pi}{2}\sgn(e_2)\int_{\partial\Omega}\phi|v_1|\partial_{\bn}v_2\,d x_1.
\end{align}

{\bf Step 3.} Note that the energy identity \eqref{energy identity when missing one direction} also holds when replacing $\be$ by $\bbf$. Adding these two identities and noticing that $\sgn(e_2)+\sgn(f_2)=0$, we get
\begin{equation}\label{curvatureangle}
\begin{aligned}
&\int_{\Omega}\frac{2\phi|v_1\nabla v_2-v_2\nabla v_1|^2}{|\bv|^2}\,dx\\
=&\int_{\Omega}\left(\angle(\be,\bi)-\angle(\be,\bv)+\angle(\bbf,\bi)-\angle(\bbf,\bv)\right)\nabla\phi\cdot(v_1\nabla v_2-v_2\nabla v_1)\,dx.
\end{aligned}
\end{equation}
If we take $\phi=\phi_R^2$ in \eqref{curvatureangle} and follow the same steps as those in the proof of Theorem~\ref{rigidity in the plane}, then it can be shown that $\bv$ must be a shear flow.  This completes the proof of Lemma \ref{rigidity lemma when missing two directions}. 
\end{proof}

The next lemma asserts that the total curvature of the flows must be finite as long as the set of flows angles for the interior domain is a proper subset of the whole circle. 

\begin{lemma}\label{finite total curvature lemma when missing one direction}
Let $\Omega=\Omega_\infty$ or $\Omega=\mbR_+^2$. Suppose that $\bv=(v_1,v_2)\in C^1(\Bar{\Omega})\cap W^{2,1}_{loc}(\Bar{\Omega})$ is a solution of \eqref{Euler} satisfying \eqref{slip boundary condition}. Assume that $v_1|_{\partial\Omega}\in L^\infty(\partial\Omega)$, and that $\bv$ satisfies \eqref{growth at horizontal far fields} and \eqref{growth at halfplane far fields}, respectively, in the respective case $\Omega=\Omega_\infty$ and $\Omega=\mbR_+^2$.
If
\begin{align}\label{condangle}
\Theta(\bv;\Omega)\subsetneqq S^1,
\end{align}
then the total curvature is finite, i.e., 
\begin{align}\label{finitecurvature}
\int_\Omega\frac{|v_1\nabla v_2-v_2\nabla v_1|^2}{|\bv|^2}\,dx<\infty.    
\end{align}
\end{lemma}

\begin{remark}
Note that the assumption \eqref{condangle} only concerns the information on flow angles in the interior of the domain. For example, if $v_1<0$ on the boundary $\partial\Omega$ and $\Theta(\bv, \Omega)=S^1\setminus \{-\bi\}$, the assumption \eqref{condangle} still holds and  the flows still have finite total curvature. 
\end{remark}
\begin{proof}[Proof of Lemma \ref{finite total curvature lemma when missing one direction}]
Let us assume that $-\be\notin \Theta(\bv;\Omega)$ for some unit vector $\be$.

{\bf Case 1: $\be$ is slant.} In this case, the energy identity \eqref{energy identity when missing one direction} holds. We shall choose the test function $\phi$ in \eqref{energy identity when missing one direction} to be $\phi_R^2$. From this, it is now clear that \eqref{finitecurvature} holds as long as the boundary term $\int_{\partial\Omega}\phi_R^2|v_1|\partial_{\bn}v_2\,d x_1$ is uniformly bounded in $R$. In the case of $\Omega=\Omega_\infty$, we use the incompressibility condition and integrate by part to have
\begin{align*}
\left|\int_{\partial\Omega}\phi_{1,R}^2|v_1|\partial_{\bn}v_2\,d x_1\right|
=\left|\int_{\mbR}\phi_{1,R}\phi_{1,R}'(v_1|v_1|(x_1,1)-v_1|v_1|(x_1,-1))\,d x_1\right|\le C\|v_1\|_{L^\infty(\partial\Omega)}^2.
\end{align*}
Similarly, uniform boundedness also holds when $\Omega=\mbR_+^2$.

{\bf Case 2: $\be$ is horizontal.}
Without loss of generality, we choose $\be=\bi$. We may assume that $\{x\in\partial\Omega| v_1(x)<0\}$ is nonempty, otherwise, it follows from Lemma \ref{rigidity lemma when missing one direction} that the total curvature is $0$. Let $\theta=\angle(\bi,\bv)$ be the angle from $\bi$ to $\bv$. The delicate point is that $\theta$ is only defined in the interior domain and may be multi-valued on the boundary where $v_1<0$. 

First, analogous to \eqref{limit energy identity}, we have the following energy identity
\begin{align}\label{energy2}
\int_{\Omega}\frac{\phi|v_1\nabla v_2-v_2\nabla v_1|^2}{|\bv|^2}\,dx=-\int_{\Omega}\theta\nabla\phi\cdot(v_1\nabla v_2-v_2\nabla v_1)\,dx \quad \text{for all } \phi\in C_0^1(\Omega).
\end{align}
 Take $\phi=\phi_1\phi_2$, where $\phi_1=\phi_1(x)\in C_c^1(\Bar{\Omega})$ and $\phi_2=\phi_2(x_2)$ are to be chosen later.
Hence  \eqref{energy2} can be written as
\begin{align}\label{energy identity when missing i}
\int_{\Omega}\frac{\phi_1\phi_2|v_1\nabla v_2-v_2\nabla v_1|^2}{|\bv|^2}\,dx
=-\int_{\Omega}\theta\phi_2\nabla\phi_{1}\cdot(v_1\nabla v_2-v_2\nabla v_1)\,dx+\sum_{i=1}^2 I_i,
\end{align}
where
\[
I_1= -\int_{\Omega}\theta\phi_1\phi_{2}^{'}v_1\partial_{x_2} v_2\,dx=\int_{\Omega}\theta\phi_1\phi_{2}^{'}v_1\partial_{x_1} v_1\,dx\ \ \text{and}\ 
\ I_2=\int_{\Omega}\theta\phi_1\phi_{2}^{'}v_2\partial_{x_2} v_1\,dx.
\]

Let us first consider the case $\Omega=\mbR_+^2$. Define 
\begin{equation*}
\phi_2(x_2)=\left\{\begin{aligned}
&\frac{x_2}{\delta}, && {\rm if}\ 0\le x_2<\delta,\\
&1, && {\rm if}\ \delta\le x_2<\infty.
\end{aligned}\right.
\end{equation*}
Since $\bv\in C^1(\Bar{\Omega})$ and $v_2=0$ on $\partial\Omega$, it is easy to see that 
\begin{align}\label{J tends to zero}
\lim_{\delta\rightarrow0}I_2=0.
\end{align}
We split $I_1$ into three integrals
\begin{align*}
I_1=I_-+I_0+I_+\vcentcolon=\left(\int_{D_-}+\int_{D_0}+\int_{D_+}\right)\left(\frac{1}{\delta}\int_{0}^{\delta}\theta\phi_1v_1\partial_{x_1} v_1\,dx_2\right)\,dx_1,
\end{align*}
where 
\[
D_0=\{x_1\in\mbR| v_1(x_1,0)=0\}\ \ \text{and } \,\,D_{\pm}=\{x_1\in\mbR| \pm v_1(x_1,0)>0\}.
\]
Similar to \eqref{J tends to zero}, one has
\begin{align}\label{I naught tends to zero}
\lim_{\delta\rightarrow0}I_0=0.
\end{align}

For every $x_1\in D_+$, $\theta$ is continuous  in some neighborhood of $(x_1,0)$ and $\theta(x_1,0)=0$. So by dominated convergence theorem, it also holds
\begin{align}\label{I positive tends to zero}
\lim_{\delta\rightarrow0}I_+=0.
\end{align}

Now we need to deal with the limit of $I_-$. First, assume that $D_-$ contains an interval, say $(a,b)$, which is maximal in the sense that $D_-$ does not contain a larger interval that contains $(a,b)$. Let $U$ be the connected component of $\{x\in\Omega| v_1(x)<0\}$ whose boundary $\partial U$ satisfies  $(a,b)\times \{0\}\subset\partial U$. By the assumption that $-\bi\notin \Theta(\bv;\Omega)$, one can conclude that $v_2\neq0$ all over $U$. If $v_2>0$ in $U$, then by the slip boundary condition \eqref{slip boundary condition}, it holds that $\partial_{x_2}v_2(x_1,0)\ge0$ for all $x_1\in(a,b)$. This, together with the incompressibility condition, implies 
\[
\partial_{x_1}v_1(x_1, 0)\leq 0 \quad \text{for all}\,\, x_1\in (a, b). 
\]
Hence $v_1(\cdot,0)$ is nonincreasing in $(a,b)$. Note that $v_1(\cdot,0)$ must be negative in the maximal interval $(a,b)$. Therefore, $b=\infty$. Similarly, if $v_2<0$ in $U$, then $D_-$ must contain an interval of the form $(-\infty,b)$.

Assume $D_-=(a,\infty)$. Note that $v_2>0$ in $U$. So for every $x_1\in (a,\infty)$, one has $\lim_{x_2\rightarrow0+}\theta(x_1,x_2)=\pi$. It follows from dominated convergence theorem that
\begin{align*}
\lim_{\delta\rightarrow0}I_-=\pi\int_{a}^{\infty}(\phi_1 v_1 \partial_{x_1}v_1)(x_1,0)\,d x_1.
\end{align*}
If $a$ is finite, then $v_1(a,0)=0$. So no matter $a$ is finite or infinite, we can always integrate by part to get
\begin{align}\label{I negative limit}
\lim_{\delta\rightarrow0}I_-=-\frac{\pi}{2}\int_{a}^{\infty}\partial_{x_1}\phi_1(x_1,0) v_1^2(x_1,0)\,d x_1.
\end{align}
Now passing to the limit $\delta\rightarrow0$ in \eqref{energy identity when missing i}, and combining \eqref{J tends to zero}, \eqref{I naught tends to zero}, \eqref{I positive tends to zero} and \eqref{I negative limit} yield that
\begin{align*}
&\int_{\Omega}\frac{\phi_1|v_1\nabla v_2-v_2\nabla v_1|^2}{|\bv|^2}\,dx\\
=&-\int_{\Omega}\theta\nabla\phi_{1}\cdot(v_1\nabla v_2-v_2\nabla v_1)\,dx
-\frac{\pi}{2}\int_{a}^{\infty}\partial_{x_1}\phi_1(x_1,0) v_1^2(x_1,0)\,d x_1
\end{align*}
for all $\phi_1\in C_c^1(\Bar{\Omega})$.
Now we can take $\phi_1=\phi_{log,R}^2$ and follow the same procedure as that for the proof of Theorem \ref{rigidity in the plane}  to get \eqref{finitecurvature}.

If $D_-=(-\infty,b)$ and $v_2<0$ in $U$, then  for every $x_1\in (-\infty,b)$, one has $\lim_{x_2\rightarrow0+}\theta(x_1,x_2)=-\pi$. If $D_-=(-\infty,b)\cup (a,\infty)$, then $v_1(a,0)=v_1(b,0)=0$, $$\lim_{x_2\rightarrow0+}\theta(x_1,x_2)=\pi
\quad \text{for every } x_1\in (a,\infty),$$
and 
$$\lim_{x_2\rightarrow0+}\theta(x_1,x_2)=-\pi\quad \text{for every }x_1\in (-\infty,b).$$ So both cases can be handled similarly.

For the case $\Omega=\Omega_\infty$, define
\begin{equation*}
\phi_2(x_2)=\left\{\begin{aligned}
&\frac{x_2+1}{\delta}, && {\rm if}\ -1\le x_2\le-1+\delta,\\
&1, && {\rm if}\ -1+\delta<x_2<1-\delta,\\
&\frac{1-x_2}{\delta}, && {\rm if}\ 1-\delta\le x_2\le1.
\end{aligned}\right.
\end{equation*}
One just needs to split the integral $I$ into the upper and lower boundary parts and estimate each part using the same method employed in the case $\Omega=\mathbb{R}^2_+$. The details are omitted here. 

Thus, the proof of Lemma \ref{finite total curvature lemma when missing one direction} is completed.
\end{proof}

The major task of this section is to establish the next proposition that characterizes Type III flows. 

\begin{pro}\label{theorem of best constant}
Let $\Omega$ be either $\Omega_\infty$ or $\mbR_+^2$. Let $\bv=(v_1,v_2)\in C^1(\Bar{\Omega})$ be a solution of \eqref{Euler} satisfying \eqref{slip boundary condition}. Assume that $\bv$ satisfies \eqref{growth at horizontal far fields} if $\Omega=\Omega_\infty$, or \eqref{growth at halfplane far fields} if $\Omega=\mbR_+^2$.

{\rm (\romannumeral1)} If $\bv$ satisfies \eqref{finitecurvature}, then the limit
\begin{align}\label{definition of D infinity}
\mathcal{J}_\infty\vcentcolon=-\lim_{R\rightarrow +\infty} \int_{\partial\Omega}\phi_R|v_1|\partial_{\bn}v_2\,dx_1
\end{align}
exists, where $\partial_{\bn}$ is the outer normal derivative. Moreover, it holds that
\begin{align}\label{lower bound of total curvature}
|\mathcal{J}_\infty|\le \frac{2}{\pi}\int_{\Omega}\frac{|v_1\nabla v_2-v_2\nabla v_1|^2}{|\bv|^2}\,dx.
\end{align}

{\rm (\romannumeral2)} If $\bv\in W^{2,1}_{loc}(\Bar{\Omega})$ and $v_1|_{\partial\Omega}\in L^\infty(\partial\Omega)$, then the following statements hold.
\begin{enumerate}[(a)]
\item  $\bv$ satisfies \eqref{finitecurvature} and equality in \eqref{lower bound of total curvature} holds if and only if one of the following three cases occurs.

Case 1. $\Theta(\bv;\Bar{\Omega})=S^1_+$.

Case 2. $\Theta(\bv;\Bar{\Omega})=S^1_-$.

Case 3. $\bv$ is a shear flow.

\item  Let $\Omega=\mbR_+^2$. If $\bv$ is of Type {\rm\uppercase\expandafter{\romannumeral 3}}, say $\Theta(\bv;\Bar{\Omega})=S^1_+$, then $v_1(\cdot,0)$ is nonincreasing in $\mbR$, and it holds that $\underline{v}_{1,-}=-\underline{v}_{1,+}>0$, and that 
\begin{align}\label{total curvature in half plane}
\int_{\Omega}\frac{|v_1\nabla v_2-v_2\nabla v_1|^2}{|\bv|^2}\,dx=\frac{\pi}{2}\underline{v}^2_{1,+}.
\end{align}
where
\[
\underline{v}_{1,-}=\lim_{x_1\to -\infty} v_1(x_1,0)\quad\text{and}\quad  \underline{v}_{1,+}=\lim_{x_1\to +\infty} v_1(x_1,0).
\]

\item  Let $\Omega=\Omega_\infty$. If $\Theta(\bv;\Bar{\Omega})=S^1_+$, then $v_1(\cdot,1)$ is nondecreasing in $\mbR$ and $v_1(\cdot,-1)$ is nonincreasing in $\mbR$, and it holds that
\begin{align}\label{boundary asymptotics in a strip}
\bar{v}_{1,+}^2-\bar{v}_{1,-}^2=\underline{v}_{1,+}^2-\underline{v}_{1,-}^2
\end{align}
and 
\begin{align}\label{total curvature in a strip}
\int_{\Omega}\frac{|v_1\nabla v_2-v_2\nabla v_1|^2}{|\bv|^2}\,dx
=\frac{\pi}{4}\{&\bar{v}_{1,+}|\bar{v}_{1,+}|-\bar{v}_{1,-}|\bar{v}_{1,-}|+\underline{v}_{1,-}|\underline{v}_{1,-}|-\underline{v}_{1,+}|\underline{v}_{1,+}|\},
\end{align}
where
\[
\underline{v}_{1,\pm}=\lim_{x_1\to \pm \infty} v_1(x_1,-1)\quad\text{and}\quad  \bar{v}_{1,\pm}=\lim_{x_1\to \pm \infty} v_1(x_1, 1).
\]
\end{enumerate}
\end{pro}

\begin{remark}
The inequality \eqref{lower bound of total curvature} can be viewed as a sort of trace inequality. Note that the total curvature, which can be written as $\int_\Omega |\nabla P|^2|\bv|^{-2}\,dx$, is a weighted homogeneous $H^1$-norm of $P$. It follows from the  Bernoulli's law (cf. \cite{Bat})  that $\frac{d}{dx_1}(P|_{\partial\Omega}+\frac{1}{2}v_1^2|_{\partial\Omega})=0$. Note that  $\mathcal{J}_\infty$ depends on the asymptotics of $v_1|v_1|$ at the boundaries, hence relating to the trace of $P$.
\end{remark}

\begin{remark}
Proposition \ref{theorem of best constant} part {\rm (\romannumeral2)}-(a) plays a crucial role in the proof of Theorem \ref{rigidity in a strip} {\rm (\romannumeral2)}. Proposition \ref{theorem of best constant} part {\rm (\romannumeral2)}-(b) and part {\rm (\romannumeral2)}-(c) tell a potential streamline pattern for a Type III flow if the flow is symmetric about axes. This will actually serve as a guide in our search for Type III flows in Section \ref{existence section}. 
\end{remark}

To prove Proposition \ref{theorem of best constant}, we shall first derive energy identities concerning the total curvature. The key observation is to test the divergence form equation \eqref{divergence-form Euler} by zeroth-degree homogeneous functions of $\bv$.

In what follows, $B$ denotes the open unit disk centered at the origin. Recall that $\bi$ denotes the unit vector $(1,0)$. For a function $g=g(y)\in C^1(S^1)$, denote $\partial_\theta g$ the derivative of $g$ with respect to arc length $\theta$. For a function $f=f(x)$, with a slight abuse of notation, $\partial_\theta f$ also denotes the angular derivative $x^\perp\cdot \nabla f(x)$.


\begin{lemma}\label{identity via homogeneous function}
Let $\Omega$ be $\mbR^2$, $\mbR_+^2$, or $\Omega_\infty$, respectively,  and $\bv=(v_1,v_2)\in C^1(\Bar{\Omega})$ be a solution of \eqref{divergence-form Euler} satisfying the slip boundary condition \eqref{slip boundary condition} (if $\partial\Omega$ is nonempty) and finite total curvature condition \eqref{finitecurvature}. Assume that $\bv$ satisfies the corresponding growth condition \eqref{growth at far fields}, \eqref{growth at halfplane far fields}, or \eqref{growth at horizontal far fields}, respectively. 
Let $g: S^1\rightarrow\mbR$ be a Lipschitz continuous and piecewise smooth function. Then it holds that
\begin{align}\label{energy identity with homogeneous G}
\int_{\Omega} \partial_\theta g\left(\frac{\bv}{|\bv|}\right)\frac{|v_1\nabla v_2-v_2\nabla v_1|^2}{|\bv|^2}\,dx
=\lim_{R\rightarrow\infty}\int_{\partial\Omega}\phi_R v_1\partial_{\bn}v_2 g\left(\sgn(v_1)\bi\right) \,dx_1,
\end{align}
where the right hand side of \eqref{energy identity with homogeneous G} vanishes  in the case $\Omega=\mbR^2$.
\end{lemma}
\begin{proof}
First, let $G\in C(\Bar{B})\cap C^\infty(B)$ be the harmonic extension of $g$ to the disk B. It follows from the  Poisson's formula that
\begin{equation*}
    G(z)=\frac{1}{2 \pi}\int_{S^1} \frac{1-|z|^2}{|z-y|^2} g(y) ds(y)\quad \text{for any}\,\, z\in B.
\end{equation*}
 Hence the straightforward computations give for every $z\in B$ that
\[
\nabla _zG(z)=-\frac{z}{\pi}\int_{S^1}\frac{g(y)}{|z-y|^2}\,ds(y)
-\frac{1-|z|^2}{\pi}\int_{S^1}\frac{g(y)(z-y)}{|z-y|^4}\,ds(y).
\]
Then integrating by part yields
\[
\partial_\theta G(z)=z^\perp\cdot\nabla_zG(z)=\frac{1-|z|^2}{2\pi}\int_{S^1}\frac{\partial_\theta g(y)}{|z-y|^2}\,ds(y).
\]
Hence there exists a constant $C>0$ such that
\begin{align}\label{outer function G}
|\partial_\theta G(z)|+(1-|z|)|\nabla_z G(z)|\le C,\ \ {\rm for\ all}\ \ z\in B.
\end{align}
By assumption, $g$ is smooth on some subset $S_*^1$ of $S^1$, where $S^1\setminus S_*^1$ consists of finitely many elements. The straightforward computations show 
\begin{align}\label{limit of tangent derivative}
\lim_{B\ni z\rightarrow y}\partial_{\theta}G(z)=\partial_\theta g(y),\ \ {\rm for\ all}\ y\in S_*^1.
\end{align}

Next, fix any $\varepsilon>0$ and $\phi\in C_c^1(\Bar{\Omega})$. Let us test \eqref{divergence-form Euler} by
\[
\phi(x) G\left(\frac{\bv(x)}{\sqrt{|\bv(x)|^2+\varepsilon}}\right)\in C_0^1(\Bar{\Omega}).
\]
Integrating by part, using the slip boundary condition \eqref{slip boundary condition},
we get the energy identity
\begin{align*}
&\int_{\Omega} \frac{\phi\partial_{\theta}G\left(\frac{\bv}{\sqrt{|\bv|^2+\varepsilon}}\right)|v_1\nabla v_2-v_2\nabla v_1|^2}{|\bv|^2+\varepsilon}\,dx\\
=&-\int_{\Omega} \frac{\varepsilon\phi \nabla\bv\nabla_zG\left(\frac{\bv}{\sqrt{|\bv|^2+\varepsilon}}\right)\cdot(v_1\nabla v_2-v_2\nabla v_1)}{(|\bv|^2+\varepsilon)^{3/2}}\,dx\\
&-\int_{\Omega}G\left(\frac{\bv}{\sqrt{|\bv|^2+\varepsilon}}\right)\nabla\phi\cdot(v_1\nabla v_2-v_2\nabla v_1) \,dx
+\int_{\partial\Omega}\phi v_1\partial_{\bn}v_2 G\left(\frac{v_1}{\sqrt{v_1^2+\varepsilon}},0\right) \,dx_1,
\end{align*}
where we have used the fact 
\begin{align*}
\nabla\left[G\left(\frac{\bv}{\sqrt{|\bv|^2+\varepsilon}}\right)\right]
=\frac{\varepsilon \nabla\bv\nabla_zG\left(\frac{\bv}{\sqrt{|\bv|^2+\varepsilon}}\right)}{(|\bv|^2+\varepsilon)^{3/2}}
+
\frac{\partial_{\theta}G\left(\frac{\bv}{\sqrt{|\bv|^2+\varepsilon}}\right)(v_1\nabla v_2-v_2\nabla v_1)}{|\bv|^2+\varepsilon},
\end{align*}

Note that $|v_1\nabla v_2-v_2\nabla v_1|=0$ a.e. on $\Omega\setminus \Omega_*$ with $\Omega_*:=\{x\in\Omega:|\bv|\neq0,\, \frac{\bv}{|\bv|}\in S_*^1\}$. Therefore, it follows from \eqref{outer function G}, \eqref{limit of tangent derivative}, and the dominated convergence theorem that one has 
\begin{align*}
\lim_{\varepsilon\rightarrow0}\int_{\Omega} \frac{\phi\partial_{\theta}G\left(\frac{\bv}{\sqrt{|\bv|^2+\varepsilon}}\right)|v_1\nabla v_2-v_2\nabla v_1|^2}{|\bv|^2+\varepsilon}\,dx
=\int_{\Omega} \frac{\phi\partial_{\theta}g\left(\frac{\bv}{|\bv|}\right)|v_1\nabla v_2-v_2\nabla v_1|^2}{|\bv|^2}\,dx,
\end{align*}
where the domain of integration is actually
$\Omega_*$.
Thanks to \eqref{outer function G}, there exists a constant $C>0$ such that
\begin{align*}
\frac{\varepsilon|\nabla_zG|\left(\frac{\bv}{\sqrt{|\bv|^2+\varepsilon}}\right)}{|\bv|^2+\varepsilon}\le C.
\end{align*}
Also, note that $G$ is smooth up to boundary points at $S_*^1$. Thus, it holds that
\begin{align*}
\lim_{\varepsilon\rightarrow0}\frac{\varepsilon|\nabla_zG|\left(\frac{\bv(x)}{\sqrt{|\bv(x)|^2+\varepsilon}}\right)}{|\bv(x)|^2+\varepsilon}=0\quad \text{for all}\,\,  x\in \Omega_*.
\end{align*}
 Applying again the dominated convergence theorem gives
\begin{align*}
\lim_{\varepsilon\rightarrow0}\int_{\Omega} \frac{\varepsilon\phi \nabla\bv\nabla_zG\left(\frac{\bv}{\sqrt{|\bv|^2+\varepsilon}}\right)\cdot(v_1\nabla v_2-v_2\nabla v_1)}{(|\bv|^2+\varepsilon)^{3/2}}\,dx
=0.
\end{align*}
The last two integrals in the energy identity can be handled easily.
So passing to the limit $\varepsilon\rightarrow0$ in the energy identity, we arrive at
\begin{align}\label{energy identity with G and phi}
&\int_{\Omega} \frac{\phi\partial_{\theta}g\left(\frac{\bv}{|\bv|}\right)|v_1\nabla v_2-v_2\nabla v_1|^2}{|\bv|^2}\,dx\nonumber\\
=&-\int_{\Omega}g\left(\frac{\bv}{|\bv|}\right)\nabla\phi\cdot(v_1\nabla v_2-v_2\nabla v_1) \,dx
+\int_{\partial\Omega}\phi v_1\partial_{\bn}v_2 g\left(\sgn(v_1)\bi\right) \,dx_1.
\end{align}

Now we choose $\phi=\phi_R$. Applying the Cauchy-Schwarz inequality and the corresponding growth condition \eqref{growth at far fields}, \eqref{growth at halfplane far fields}, or \eqref{growth at horizontal far fields}, one has
\begin{align*}
&\left|\int_{\Omega}g\left(\frac{\bv}{|\bv|}\right)\nabla\phi_R\cdot(v_1\nabla v_2-v_2\nabla v_1) \,dx\right|\\
\le&\|g\|_{L^\infty(S^1)}\left(\int_{\Omega}|\bv|^2|\nabla\phi_R|^2\,dx\right)^{\frac12}
\left(\int_{\{|\nabla\phi_R|\neq0\}}\frac{|v_1\nabla v_2-v_2\nabla v_1|^2}{|\bv|^2} \,dx\right)^{\frac12}\\
\le&C\|g\|_{L^\infty(S^1)}
\left(\int_{\{|\nabla\phi_R|\neq0\}}\frac{|v_1\nabla v_2-v_2\nabla v_1|^2}{|\bv|^2} \,dx\right)^{\frac12}.
\end{align*}
This, together with the finite total curvature condition
\eqref{finitecurvature}, gives
\[
\lim_{R\to \infty}\left|\int_{\Omega}g\left(\frac{\bv}{|\bv|}\right)\nabla\phi_R\cdot(v_1\nabla v_2-v_2\nabla v_1) \,dx\right|=0.
\]
Therefore, sending $R$ to $\infty$ in \eqref{energy identity with G and phi} gives \eqref{energy identity with homogeneous G}.
This completes the proof of the lemma.
\end{proof}

A consequence of Lemma \ref{identity via homogeneous function} is the following corollary.

\begin{coro}
\label{lemma from lemma}
Let $\Omega$ be $\Omega_\infty$ or $\mbR_+^2$, respectively, and $\bv=(v_1,v_2)\in C^1(\Bar{\Omega})$ be a solution of the steady Euler system \eqref{divergence-form Euler} satisfying the slip boundary condition \eqref{slip boundary condition} and finite total curvature condition \eqref{finitecurvature}. Assume that $\bv$ satisfies \eqref{growth at horizontal far fields} and \eqref{growth at halfplane far fields} in the case $\Omega=\Omega_\infty$ and $\Omega=\mbR_+^2$, respectively.  Let $g: S^1\rightarrow\mbR$ be any Lipschitz continuous and piecewise smooth function. Then the following statements hold:

{\rm (\romannumeral1)}
$\lim_{R\rightarrow\infty}\int_{\partial\Omega}\phi_R v_1\partial_{\bn}v_2\,dx_1=0$.

{\rm (\romannumeral2)} 
If $g(\bi)=g(-\bi)$, then
\begin{align*}
\int_{\Omega} \partial_{\theta}g\left(\frac{\bv}{|\bv|}\right)\cdot\frac{|v_1\nabla v_2-v_2\nabla v_1|^2}{|\bv|^2}\,dx=0.
\end{align*}

{\rm (\romannumeral3)} $\mathcal{J}_\infty$
defined in \eqref{definition of D infinity}  exists and satisfies that
\begin{align}\label{D infty formula}
\mathcal{J}_\infty=-\int_{\Omega} \partial_{\theta}g\left(\frac{\bv}{|\bv|}\right)\cdot\frac{|v_1\nabla v_2-v_2\nabla v_1|^2}{|\bv|^2}\,dx
\end{align}
whenever $g(\pm\bi)=\pm 1$.   
\end{coro}
\begin{proof}
First, choosing $g\equiv1$ in \eqref{energy identity with homogeneous G} proves Part {\rm (\romannumeral1)}. Part {\rm (\romannumeral2)} follows from \eqref{energy identity with homogeneous G} and Part {\rm (\romannumeral1)}. Part {\rm (\romannumeral3)} is an immediate consequence of \eqref{energy identity with homogeneous G}. This completes the proof.
\end{proof}

We are now in a position to prove Proposition \ref{theorem of best constant}.

\begin{proof}[Proof of Proposition \ref{theorem of best constant}]
The proof is divided into four steps, each dealing with a part of the proposition.

{\bf Step 1.} The existence of $\mathcal{J}_\infty$ has been proved in Corollary \ref{lemma from lemma} {\rm (\romannumeral3)}. We directly get from \eqref{D infty formula} that
\begin{align*}
|\mathcal{J}_\infty|\le \|\partial_{\theta}g\|_{L^\infty(S^1)}\int_{\Omega} \frac{|v_1\nabla v_2-v_2\nabla v_1|^2}{|\bv|^2}\,dx
\end{align*}
for every Lipschitz continuous and piecewise smooth function $g:S^1\rightarrow\mbR$ satisfying $g(\pm\bi)=\pm1$.
It follows from the mean value theorem that such a function $g$ must satisfy 
\[
\|\partial_{\theta}g\|_{L^\infty(S^1)}\ge\frac{2}{\pi}.
\]
This lower bound is attainable by simply choosing $g$ such that $\partial_\theta g=-\frac{2}{\pi}$ on $S^1_+$ and $\partial_\theta g=\frac{2}{\pi}$ on $S^1_-$ ($S^1$ is oriented in the anticlockwise direction). With such a choice of $g$ in \eqref{D infty formula}, we arrive at
\begin{align}\label{explicit formula for D}
\mathcal{J}_\infty=\frac{2}{\pi}\int_{\Omega}\sgn(v_2)\frac{|v_1\nabla v_2-v_2\nabla v_1|^2}{|\bv|^2}\, dx.
\end{align}
This immediately gives rise to \eqref{lower bound of total curvature}.

{\bf Step 2.}
We prove Part (a) in Proposition \ref{theorem of best constant} {\rm (\romannumeral2)}. We first prove the following stronger statement that if $\Theta(\bv;\Bar{\Omega})\subsetneqq S^1$, then the total curvature is finite and equality in \eqref{lower bound of total curvature} holds. For this,  assume $-\be\notin \Theta(\bv;\Bar{\Omega})$ for some unit vector $\be$. If $\be$ is horizontal, then it follows from Proposition \ref{rigidity lemma when missing one direction} that $\bv$ is a shear flow. So the total curvature is zero and equality in \eqref{lower bound of total curvature} holds. If $\be$ is slant, we know from Lemma \ref{finite total curvature lemma when missing one direction} that the total curvature must be finite. Then choosing $\phi=\phi_R$ in \eqref{energy identity when missing one direction} and letting $R$ tend to $\infty$, we get
\begin{align*}
\int_{\Omega}\frac{|v_1\nabla v_2-v_2\nabla v_1|^2}{|\bv|^2}\,dx
=\frac{\pi}{2}\sgn(e_2)\mathcal{J}_\infty.
\end{align*}
This implies that equality in \eqref{lower bound of total curvature} still holds.

Conversely, if equality in \eqref{lower bound of total curvature} holds, we need to show that $\Theta(\bv;\Bar{\Omega})=S^1_+$, or $\Theta(\bv;\Bar{\Omega})=S^1_-$, or $\bv$ is a shear flow. Without loss of generality, one assumes
\begin{align*}
\mathcal{J}_\infty=\frac{2}{\pi}\int_{\Omega}\frac{|v_1\nabla v_2-v_2\nabla v_1|^2}{|\bv|^2}\,dx.
\end{align*}
This, together with \eqref{explicit formula for D}, immediately implies that $v_1\nabla v_2-v_2\nabla v_1\equiv0$ in $\{v_2\le0\}$. If $\{v_2<0\}$ is not empty, we can argue as in the proof of Lemma \ref{trivial curvature lemma} to show $v_2\equiv0$ in $\{v_2<0\}$, and lead to a contradiction. So one has $v_2\ge0$ in $\Omega$, which says $\Theta(\bv,\Bar{\Omega})\subset S^1_+$. If $\Theta(\bv,\Bar{\Omega})= S^1_+$, the proof is done. Finally, if $\Theta(\bv,\Bar{\Omega})\subsetneqq S^1_+$, it follows from Propositions \ref{rigidity lemma when missing one direction} and \ref{rigidity lemma when missing two directions} that $\bv$ must be a shear flow. This completes the proof of Part (a) of {\rm (\romannumeral2)}.

{\bf Step 3.} Suppose that $\Omega=\mbR_+^2$ and $\Theta(\bv;\Bar{\Omega})=S^1_+$. Since $v_2\ge0$ in $\Bar{\Omega}$, it follows from the slip boundary condition \eqref{slip boundary condition} that $\partial_{x_2}v_2(x_1,0)\ge0$ for all $x_1\in\mbR$. Thanks to the incompressibility condition, one has $\partial_{x_1}v_1(x_1,0)\le0$. Thus  $v_1(\cdot,0)$ is nonincreasing in $\mbR$. It follows from Corollary \ref{lemma from lemma} {\rm (\romannumeral1)} and direct computations that $\underline{v}_{1,-}^2=\underline{v}_{1,+}^2$.
Next, we get from \eqref{explicit formula for D} that
\begin{align*}
\int_{\Omega}\frac{|v_1\nabla v_2-v_2\nabla v_1|^2}{|\bv|^2}\,dx
=\frac{\pi}{2}|\mathcal{J}_\infty|=\frac{\pi}{4}\big|\underline{v}_{1,-}|\underline{v}_{1,-}|-\underline{v}_{1,+}|\underline{v}_{1,+}|\big|.
\end{align*}
The total curvature for a Type {\rm\uppercase\expandafter{\romannumeral 3}} steady flow must be strictly positive. Hence we conclude that $\underline{v}_{1,-}=-\underline{v}_{1,+}>0$ and \eqref{total curvature in half plane} holds. This completes the proof of Part (b) of {\rm (\romannumeral2)}.

{\bf Step 4.} Suppose that $\Omega=\Omega_\infty$ and $\Theta(\bv;\Bar{\Omega})=S^1_+$. The proof is similar to that in Step 3. Since $v_2\ge0$ in $\Omega$, the slip boundary condition and the incompressibility condition imply that $v_1(\cdot,1)$ is nondecreasing in $\mbR$ and $v_1(\cdot,-1)$ is nonincreasing in $\mbR$. Hence the identity \eqref{boundary asymptotics in a strip} follows from Corollary \ref{lemma from lemma} {\rm (\romannumeral1)}, and the identity \eqref{total curvature in a strip} follows from \eqref{explicit formula for D} and the monotonicity of $v_1(\cdot,1)$ and $v_1(\cdot,-1)$.

This finishes the proof of Proposition \ref{theorem of best constant}.
\end{proof}


\section{Classification, stability, and quantitative symmetry}\label{section rigidity in strip}

In this section, building upon the results from Section \ref{section steady flows with finite total curvature}, we establish Theorem~\ref{rigidity in a strip}, Corollary \ref{stability of P flow}, and a quantitative version of symmetry for the set of flow angles.

We start with the proof of Theorem~\ref{rigidity in a strip}.

\begin{proof}[Proof of Theorem~\ref{rigidity in a strip}]
The proof of Theorem~\ref{rigidity in a strip} {\rm (\romannumeral1)}  is similar to that of Theorem~\ref{rigidity in the plane} and Proposition \ref{rigidity lemma when missing one direction}, so we omit it here.

To prove Theorem~\ref{rigidity in a strip} {\rm (\romannumeral2)}, let us assume $\Theta(\bv,\Bar{\Omega})\subsetneqq S^1$. By Lemma \ref{finite total curvature lemma when missing one direction}, $\bv$ has finite total curvature. As proved in Step 2 of the proof of Proposition \ref{theorem of best constant}, equality in \eqref{lower bound of total curvature} holds. Then Proposition \ref{theorem of best constant} implies that $\Theta(\bv;\Bar{\Omega})=S^1_+$, or $\Theta(\bv;\Bar{\Omega})=S^1_-$, or $\bv$ is a shear flow.

Finally, Theorem~\ref{rigidity in a strip} {\rm (\romannumeral3)} is an immediate consequence of Lemmas \ref{finite total curvature lemma when missing one direction} and \ref{trivial curvature lemma}.  
So the proof of Theorem~\ref{rigidity in a strip} is completed.
\end{proof}

Next, we prove the stability results stated in Corollary \ref{stability of P flow}.

\begin{proof}[Proof of Corollary \ref{stability of P flow}]
{\rm (\romannumeral1)} Since $\bv(x_1-ct,x_2)$ is a traveling wave solution of the unsteady Euler system, $\bv_c\vcentcolon=(v_1-c,v_2)$ is a  solution of the steady Euler system  \eqref{Euler}. Hence one has
\[
(v_1-c)\partial_{x_1}\omega+v_2\partial_{x_2}\omega=0.
\]
where $\omega$ is the vorticity of the flow $\bv$. 

It follows from the assumption \eqref{smallinCor} that $\partial_{x_2}\omega<0$ in $\Bar{\Omega}$. This implies $v_2=0$ whenever $v_1-c=0$. It immediately follows that $\Theta(\bv_c;\Bar{\Omega})\subset S^1\setminus\{(0,1),(0,-1)\}$. By Theorem \ref{rigidity in a strip} {\rm (\romannumeral2)} (or Proposition \ref{rigidity lemma when missing two directions}), $\bv_c$ is a shear flow, and so is $\bv$.

{\rm (\romannumeral2)}
The proof is similar to that for part (\romannumeral1). First, it follows from the equation \eqref{divergence-form Euler} that $v_1\Delta v_2=v_2\Delta v_1$. When $v_1(x)=0$, by the assumption \eqref{smallincor_2}, $x_2$ must be near some zero of $s$. This, together with the fact that
\[
\Delta v_1(x)=\Delta(v_1-s)(x)+s''(x_2)\neq0, \]
yields
$v_2(x)=0$. As before, it follows that $\Theta(\bv;\Bar{\Omega})\subset S^1\setminus\{(0,1),(0,-1)\}$.
Also, thanks to the incompressibility condition and the slip boundary condition, $v_2$ must be bounded. So \eqref{growth at horizontal far fields} is satisfied. Recalling again Theorem~\ref{rigidity in a strip} {\rm (\romannumeral2)} finishes the proof.
\end{proof}

Theorems \ref{rigidity in the plane} and \ref{rigidity in a strip} say that the set of flow angles of any non-shear steady Euler flows can only be the whole circle or semicircle. Next, we quantify this important property by showing a sort of equal distribution property for the total curvature.
 For every $y\in S^1$, define
\begin{align*}
\Sigma_y\vcentcolon=\left\{x\in\Omega: |\bv(x)|>0, \frac{\bv(x)}{|\bv(x)|}=y \right\}.
\end{align*}
If the total curvature is finite, then the coarea formula (cf. \cite[Chapter 3]{EG}) yields that
\begin{align*}
\mfk(y):=\int_{\Sigma_y}|\bv|^2\left|\nabla\left(\frac{\bv}{|\bv|}\right)\right|\,d\mcH^1 \in L^1(S^1),
\end{align*}
 and the total curvature can be written as
\begin{align}\label{coarea1}
\int_\Omega |\bv|^2\left|\nabla\left(\frac{\bv}{|\bv|}\right)\right|^2\,dx=\int_{S^1}\mfk(y)\,d \theta(y).
\end{align}

\begin{pro}\label{total curvature evenly distributed}
Let $\Omega$ be $\mbR^2$, $\mbR_+^2$, or $\Omega_\infty$, and $\bv=(v_1,v_2)\in C^1(\Bar{\Omega})$ be a solution of \eqref{divergence-form Euler} satisfying \eqref{slip boundary condition} (if $\partial\Omega$ is nonempty) and \eqref{finite total curvature}. Assume that $\bv$ satisfies the corresponding growth constraints \eqref{growth at far fields}, \eqref{growth at halfplane far fields}, or \eqref{growth at horizontal far fields}. Then the following identities hold:

{\rm (\romannumeral1)} If $\Omega=\mbR^2$, then
\begin{align*}
\mfk(y)=\frac{1}{2\pi}\int_\Omega |\bv|^2\left|\nabla\left(\frac{\bv}{|\bv|}\right)\right|^2\,dx,\ \ a.e.\ y\in S^1.
\end{align*}

{\rm (\romannumeral2)} If $\Omega=\Omega_\infty$ or $\Omega=\mbR_+^2$, then
\begin{align*}
\mfk(y)=\frac{1}{\pi}\int_{\{x\in\Omega|\pm v_2\ge0\}}|\bv|^2\left|\nabla\left(\frac{\bv}{|\bv|}\right)\right|^2\,dx,\ \ a.e.\ y\in S^1_{\pm}.
\end{align*}
\end{pro}


\begin{proof}
If $\Omega=\mbR^2$, we apply Lemma \ref{identity via homogeneous function} and the coarea formula to get
\begin{align}\label{coarea2}
\int_{S^1}\partial_\theta g(y)\mfk(y)\,d\theta(y)=0\quad \text{for all } g\in C^\infty(S^1).
\end{align}
 It immediately follows that
\[
\mfk(y)={\rm const},\ \ {\rm for\ a.e.}\  y\in S^1.
\]
This, together with \eqref{coarea1}, proves Part  (\romannumeral1) of Proposition~\ref{total curvature evenly distributed}.

If $\Omega=\Omega_\infty$ or $\Omega=\mbR_+^2$, by Corollary \ref{lemma from lemma} (\romannumeral2), \eqref{coarea2} holds for all $g\in C^\infty(S^1)$ with $g(\bi)=g(-\bi)$. In particular, it holds for all $g\in C^\infty(S^1)$ supported on $S^1_+$ or $S^1_-$. 
Therefore, $\mfk(y)$ is a constant for $y\in S_+^1$ and $y\in S_-^1$, respectively. 
Hence Part (\romannumeral2) of Proposition ~\ref{total curvature evenly distributed} is proved.
\end{proof}



\section{Existence of Type {\thr} flows}\label{existence section}

This section is devoted to the proof of the existence of Type {\thr} flows, i.e., Theorem~\ref{existence_thm}. The construction is based on seeking a positive solution to some semilinear elliptic equation
\begin{equation}\label{stream function of type3}
\left\{\begin{aligned}
&-\Delta u=f(u), && {\rm in}\ \Omega,\\
&u=0, && {\rm on}\ \partial\Omega,
\end{aligned}\right.
\end{equation}
where $\Omega$ is a subset of $\mbR_+^2$ or $\Omega_\infty$. 

\subsection{Classical monotonicity and symmetry results}
Let $Q\vcentcolon=(0,\infty)^2$ be the first quadrant of the plane and $E_\infty\vcentcolon=(0,\infty)\times(-1,1)$ the half-infinite strip. We first recall a monotonicity and symmetry property for positive solutions to \eqref{stream function of type3}.

\begin{lemma}\label{lem6.1} Assume that $f\in C^1(\mathbb{R})$.
Let $u\in C^2(\Bar{\Omega})$ be a positive and bounded solution to \eqref{stream function of type3}. If $\Omega=Q$, it holds that
\begin{align}\label{monotonicity in first quardrant}
\partial_{x_1}u>0\ {\rm and}\ 
\partial_{x_2}u>0\ {\rm in}\ Q.
\end{align}
If $\Omega=E_\infty$, it holds that
\begin{align}
&\partial_{x_1}u>0\ {\rm in}\ E_\infty,\label{monotonicity in half strip}\\
&\partial_{x_2}u>0\ {\rm in}\ (0,\infty)\times(-1,0)\label{monotonicity in a quarter strip},
\end{align}
and
\begin{align}\label{symmetry in half strip}
u(x_1,-x_2)=u(x_1,x_2)\ {\rm for}\ x\in E_\infty.
\end{align}
\end{lemma}
\begin{proof}
The estimate
\eqref{monotonicity in half strip} is due to \cite[Theorem 1.3]{BN}. While \eqref{monotonicity in first quardrant}, \eqref{monotonicity in a quarter strip} and \eqref{symmetry in half strip} are essentially due to \cite[Theorem $1.1^\prime$]{BCN}. Taking \eqref{monotonicity in a quarter strip} for example, one can check that \cite[Theorem 1.6]{BCN96} still holds when the domain there is replaced by $(0,\infty)\times(-1,-1+\delta)$ with sufficiently small $\delta$. So as in the proof of \cite[Theorem $1.1^\prime$]{BCN}, the moving plane method can be initiated. Due to the homogeneous Dirichlet boundary condition for $u$, one does not even need \cite[Lemma 2.6]{BCN} to push the moving planes upwards until they reach the positive $x_1$-axis. This completes the proof of the lemma.
\end{proof}

\begin{remark}
Assuming $f(0)\ge0$ in Lemma \ref{lem6.1}, then the Hopf Lemma applies, so that if $\Omega=Q$, one has
\begin{align}\label{boundary velocity1}
\partial_{x_1}u(0,x_2)>0\ {\rm for}\ x_2>0,\ {\rm and}\ \partial_{x_2}u(x_1,0)>0\ {\rm for}\ x_1>0,
\end{align}
and that if $\Omega=E_\infty$, one has
\begin{align}\label{boundary velocity2}
\partial_{x_1}u(0,x_2)>0\ {\rm for}\ -1<x_2<1,\ {\rm and}\ \pm\partial_{x_2}u(x_1,\pm1)<0\ {\rm for}\ x_1>0.
\end{align}
\end{remark}

\subsection{Existence of Type-III flows  in half-plane}
Let $f\in C^2(\mbR)$ be an odd function, and satisfy
\[
f(\pm1)=0, \quad f'(0)>0, \quad f'(\pm1)<0,
\]
and that $f(s)/s$ is strictly decreasing on $(0,1)$. Note that $f(s)=s-s^3$ (which corresponds to the so called Allen-Cahn equation) is a particular case.
It was proved in \cite{DFP} that there exists a unique solution $u\in C^\infty(\Bar{Q})$ to the problem \eqref{stream function of type3} with $\Omega=Q$ such that $0<u<1$ in $Q$. 
Let
\begin{equation}\label{oddextension}
    \tilde{u}(x_1,x_2) =\left\{
    \begin{aligned}
        & u(x_1, x_2),\quad &&\text{if}\,\, (x_1,x_2)\in Q,\\
        & -u(-x_1, x_2),\quad &&\text{if}\,\, (-x_1,x_2)\in Q.
    \end{aligned}
    \right.
\end{equation}
Then $\Tilde{u}$ solves the elliptic equation in \eqref{stream function of type3} in $\mbR_+^2$. Denote $\bv=\nabla^\perp \Tilde{u}=(-\partial_{x_2}\Tilde{u},\partial_{x_1}\Tilde{u})$. Note that $|\bv|=|\nabla \Tilde{u}|$ is bounded by a standard gradient estimate. Then $\bv$ is a bounded and smooth solution to the Euler system \eqref{Euler} in $\mbR_+^2$ satisfying the slip boundary condition \eqref{slip boundary condition}. It follows from Theorem \ref{rigidity in a strip} {\rm (\romannumeral2)} and \eqref{monotonicity in first quardrant} that $\bv$ must be of Type {\thr}. Note that by uniqueness, $u$ is symmetric with respect to the line $x_2=x_1$. Moreover, by \cite[Lemma 6]{DFP}, $\bv$ converges to a shear flow uniformly for $x_2$ (resp. $x_1$) in bounded intervals as $x_1\rightarrow+\infty$ (resp. $x_2\rightarrow+\infty$). Incorporating all the information and remembering \eqref{monotonicity in first quardrant} and \eqref{boundary velocity1}, the Type {\thr} flow constructed above exhibits streamlines shown in Figure \ref{figure of type three flow in half plane}. This flow has exactly one stagnation point, which is also a saddle point of streamlines, located at the origin.

\begin{figure}[h]
\centering
\includegraphics[height=6cm, width=14cm]{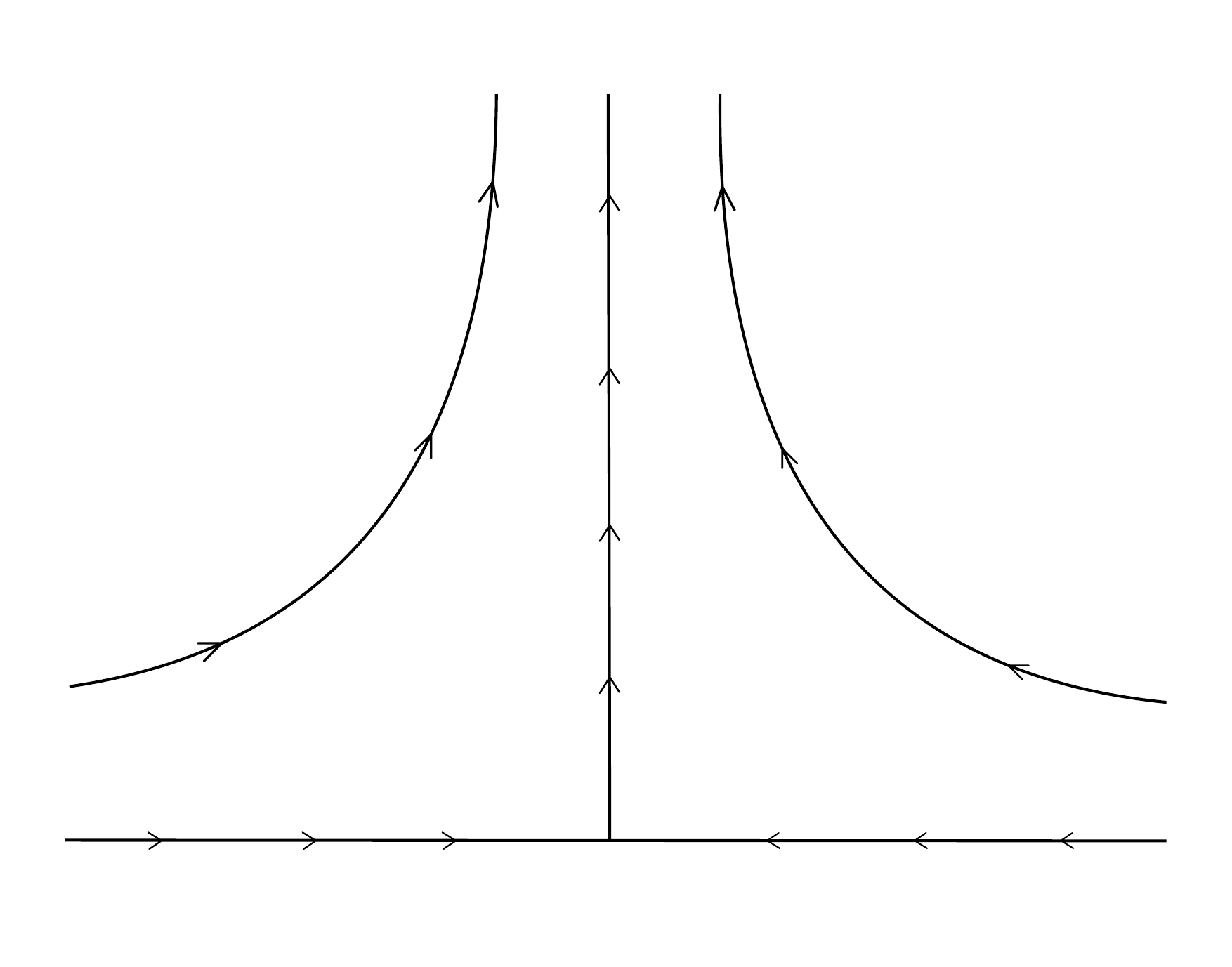}
\caption{Type {\thr} flow in $\mbR_+^2$}
\label{figure of type three flow in half plane}
\end{figure}

Obviously, by an odd reflection about the $x_1$-axis, the solution $\Tilde{u}$ defined in \eqref{oddextension} can be further extended to an entire solution to \eqref{stream function of type3} in $\mbR^2$. Such a solution is called a saddle solution with four ends, which is a special case of the so-called finite-end solutions. Roughly speaking, an end is a direction toward which the flow converges to a shear flow. For more studies on the existence, qualitative properties, and classification of finite-end solutions to the Allen-Cahn equation in two-dimensional case, one may refer to \cite{ACM,DKP,DKPW,Gui0,Gui,GLW,KLP1,KLP2,KLP3,KLPW,WW} and references therein.

\subsection{Existence of Type-III flows  in a strip}

We shall construct the stream function of a Type {\thr} flow in $\Omega_\infty$ by solving a bounded solution $u$ to the problem
\begin{eqnarray}\label{slee in a strip}
\left\{\begin{aligned}
&-\Delta u=f(u), & {\rm in}\ \Omega_\infty,\\
&u=0, & {\rm on}\ \partial\Omega_\infty,\\
&x_1 u>0, & {\rm in}\ \{x\in\Omega_\infty| x_1\neq0\}.
\end{aligned}\right.
\end{eqnarray}
In this subsection, let $f\in C^1(\mbR)$ satisfy
\begin{equation}\label{fproperty1}
    f(s)=-f(-s),\quad |f(s)|\le M\quad \text{for all}\,\, s\in \mathbb{R} \,\, \text{and some }M>0, \,\, 
    \end{equation}
    together with
    \begin{equation}\label{fproperty2}
 f'(0)>\frac{\pi^2}{4}, \quad \text{and}\,\,  \left(\frac{f(s)}{s}\right)'>0 \,\,\text{for }s>0.
\end{equation}
 A typical example of such functions is $f(s)= \lambda\arctan(s)$ with $\lambda>\frac{\pi^2}{4}$. 

The restriction of a solution to \eqref{slee in a strip} to $E_\infty$ is a positive solution to \eqref{stream function of type3} in $E_\infty$. Conversely, given a positive solution $u$ of the problem \eqref{stream function of type3} in $\Omega=E_\infty$. If we extend $u$ to be $\tilde{u}$ via the extension in \eqref{oddextension} with $Q$ replaced by $E_\infty$ and still denote $\tilde{u}$ by $u$, then ${u}$ is a solution to \eqref{slee in a strip}.

We shall use the method of sub- and supersolutions to construct a positive and bounded solution to \eqref{stream function of type3} in $E_\infty$. 
In the following, a substantial portion of our argument is standard and can be located in references such as \cite{BHR,BL,DFP}. Nevertheless, for the reader's convenience, we provide essential elaborations in our proof.

We start with the existence and uniqueness of a one-dimensional solution.

\begin{lemma}\label{lemma of existence and uniqueness of 1d solution}
Assume that $f\in C^1(\mbR)$ satisfies \eqref{fproperty1} and \eqref{fproperty2}. 
There exists a unique positive solution $\Bar{u}\in C^2([-1,1])$ to the problem
\begin{eqnarray}\label{1d solution}
\left\{\begin{aligned}
&-\Bar{u}''=f(\Bar{u}), & {\rm in}\ (-1,1),\\
&\Bar{u}(\pm1)=0.
\end{aligned}\right.
\end{eqnarray}
\end{lemma}
\begin{proof}
Since $f(0)=0$ and $f'(0)>\frac{\pi^2}{4}$, for $\varepsilon$ sufficiently small, one has $f(s)\ge \frac{\pi^2}{4}s$ for all $s\in[0,\varepsilon]$. Hence for the same $\varepsilon$, the function $\varphi_\varepsilon(x_2)\vcentcolon=\varepsilon\cos(\frac{\pi}{2}x_2)$ satisfies 
\begin{equation*}
    \left\{\begin{aligned}
&-{\varphi}''_{\varepsilon}\leq f(\varphi_\varepsilon), & {\rm in}\ (-1,1),\\
&\varphi_\varepsilon(\pm1)=0.
\end{aligned}\right.
\end{equation*}
Therefore, $\varphi_\varepsilon$
is a subsolution to the problem \eqref{1d solution}. Note that $|f|\le M$. So for every $l\ge M$, $\psi_l(x_2)\vcentcolon=\frac12 l(1-x_2^2)$ is a supersolution. Now for $\varepsilon$ small, we have $\varphi_\varepsilon\le\psi_l$ in $[-1,1]$. Then by the method of sub- and supersolutions (see, e.g., \cite{Evans}), there exists a solution $\Bar{u}$ satisfying $\varphi_\varepsilon\le\Bar{u}\le\psi_l$ in $[-1,1]$. Obviously, $\Bar{u}$ is positive in $(-1,1)$.

Assume $\Bar{v}\in C^2([-1,1])$ is also a positive solution. Then the function $\max\{\Bar{u},\Bar{v}\}$ is a subsolution. Let $l$ be so large that $\max\{\Bar{u},\Bar{v}\}\le \psi_l$ in $[-1,1]$. Then there exists a third positive solution $\Bar{w}$ between $\max\{\Bar{u},\Bar{v}\}$ and $\psi_l$. Using the elliptic equation and integrating by part yield
\begin{align*}
0=\int_{-1}^{1}(\Bar{u}\Bar{w}''-\Bar{w}\Bar{u}'')\,dx_2=\int_{-1}^{1}\Bar{u}\Bar{w}\left(\frac{f(\Bar{u})}{\Bar{u}}-\frac{f(\Bar{w})}{\Bar{w}}\right)\,dx_2.
\end{align*}
Since $0<\Bar{u}\le\Bar{w}$ in $(-1,1)$, and $\frac{f(s)}{s}$ is strictly decreasing, we conclude that $\Bar{u}=\Bar{w}$. Similarly, we have $\Bar{v}=\Bar{w}$. This completes the proof.
\end{proof}

In what follows, for $k\ge1$, we denote $\Omega_k$ and $E_k$ the open rectangles 
 $(-k,k)\times(-1,1)$ and $(0,k)\times(-1,1)$, respectively.

\begin{lemma}[Existence]\label{lemma of existence}
There exists a bounded solution $u\in C^2(\overline{\Omega_\infty})$ to \eqref{slee in a strip}.
\end{lemma}
\begin{proof}
First, the one-dimensional solution $\Bar{u}$ constructed in Lemma \ref{lemma of existence and uniqueness of 1d solution} is a supersolution to \eqref{stream function of type3} in $E_\infty$. Next, we need to construct a nonnegative subsolution. Note that there exists a small positive number $\delta$ such that 
\[
f'(0)>\frac{\pi^2}{4(1-\delta)^2}+\delta^2.
\]
So for $\varepsilon$ small, it holds that 
\[
f(s)\ge \left(\frac{\pi^2}{4(1-\delta)^2}+\delta^2\right)s\quad \text{for}\,\, s\in[0,\varepsilon].
\]
Define
\begin{eqnarray}\label{definition of subsolution}
\ubar{u}(x)=\ubar{u}(x;h,\delta)\vcentcolon=\left\{\begin{aligned}
&\varepsilon\sin(\delta (x_1-h))\cos\left(\frac{\pi}{2(1-\delta)}x_2\right), && x\in E_{h,\delta},\\
&0, && x\in\overline{E_\infty}\setminus E_{h,\delta},
\end{aligned}\right.
\end{eqnarray}
where $E_{h,\delta}\vcentcolon=(h,h+\frac{\pi}{\delta})\times(-(1-\delta),1-\delta)$ and $h\ge0$ (it suffices to choose $h=0$ in this proof, but we need to choose $h$ large enough in the proof of uniqueness later). The straightforward computations yield that $0\le \ubar{u}\le \varepsilon$, and 
\[
-\Delta\ubar{u}=\left(\frac{\pi^2}{4(1-\delta)^2}+\delta^2\right)\ubar{u}\le f(\ubar{u})\quad \text{in } E_{h,\delta}.
\]
It is easy to see the outer normal derivative $\partial_{\bn}\ubar{u}\le0$ on $\partial E_{h,\delta}$. Therefore, for every $\phi\in C_c^1(E_\infty)$ with $\phi\ge0$, it holds that
\begin{align*}
\int_{E_\infty}\nabla\ubar{u}\cdot\nabla\phi\,dx=&\int_{E_{h,\delta}}\nabla\ubar{u}\cdot\nabla\phi\,dx=-\int_{E_{h,\delta}}\Delta\ubar{u}\phi\,dx+\int_{\partial E_{h,\delta}}\partial_{\bn}u\phi\,ds\\
\le&\int_{E_{h,\delta}}f(\ubar{u})\phi\,dx=\int_{E_{\infty}}f(\ubar{u})\phi\,dx.
\end{align*}
Hence $\ubar{u}$ is a weak subsolution to \eqref{stream function of type3} in $E_\infty$. By our construction, it is easily seen that $\ubar{u}\le \Bar{u}$ in $E_\infty$ provided $\varepsilon$ is small enough, where $\bar{u}$ is the solution of \eqref{1d solution} established in Lemma \ref{lemma of existence and uniqueness of 1d solution}.

Now for every $k$ with $k+1\ge h+\frac{\pi}{\delta}$, $\ubar{u}$ and $\Bar{u}$ are still sub- and supersolutions to \eqref{stream function of type3} in $E_{k+1}$. So there exists a weak solution $u_k\in H_0^1(E_{k+1})$ to \eqref{stream function of type3} in $E_{k+1}$ such that $\ubar{u}\le u_k\le \Bar{u}$ in $E_{k+1}$. By the standard regularity theory for elliptic equations (cf. \cite{GT}), $u_k$ is regular up to the boundary except possibly at the corners of $E_{k+1}$. Since $f$ is an odd function, the odd extension of $u_k$ with respect to $x_1$, denoted by $\Tilde{u}_k$, is a solution to \eqref{stream function of type3} in $\Omega_{k+1}$. Hence for any $\alpha\in(0,1)$, it follows again from the regularity theory for elliptic equations that
\[
\|\Tilde{u}_k\|_{C^{2,\alpha}(\overline{\Omega_k})}\le C,
\]
where $C$ is a constant depending on $\alpha$ but independent of $k$. By the Arzel\'a–Ascoli theorem and Cantor's diagonal argument, $u_k$ (up to a subsequence) converges to some limit $u$ in $C^2$ on any compact subsets of $\overline{\Omega_\infty}$. Hence $u\in C^2(\overline{\Omega_\infty})$ is a bounded solution to \eqref{stream function of type3} in $\Omega_\infty$ and $u$ is odd in $x_1$. Since $u\ge\ubar{u}\ge0$ in $E_\infty$ and $\ubar{u}$ is nontrivial, the strong maximum principle implies that $u$ must be positive in $E_\infty$. Therefore, $u$ solves \eqref{slee in a strip}. This completes the proof of the lemma.
\end{proof}

\begin{lemma}[Asymptotics]\label{lemma of asymptotics}
Let $u\in C^2(\overline{\Omega_\infty})$ be any bounded solution to \eqref{slee in a strip}. It holds that
\begin{align}\label{asymptotics of gradient u}
\lim_{x_1\rightarrow+\infty}(|u(x_1,x_2)-\Bar{u}(x_2)|+|\nabla u(x_1,x_2)-(0,\Bar{u}'(x_2))|)=0
\end{align}
{uniformly for }$x_2\in[-1,1]$.
\end{lemma}
\begin{proof}
First, the restriction of $u$ to $\overline{E_\infty}$ is a positive and bounded solution to \eqref{stream function of type3} in $E_\infty$. Thanks to \eqref{monotonicity in half strip}, there exists a bounded function, $u^\sharp(x_2)$, defined by 
\[
u^\sharp(x_2)=\lim_{x_1\rightarrow+\infty}u(x_1,x_2)\quad \text{for all }x_2\in[-1,1].
\]
Obviously, $u^\sharp$ is positive in $(-1,1)$ and $u^\sharp(\pm1)=0$.

For $t\ge0$, define $u^{(t)}(x)=u(x_1+t,x_2)$ for $x\in\overline{E_1}$. It follows from the regularity theory for elliptic equations that for any $\alpha\in(0,1)$, one has
\[
\sup_{t\ge0}\|u^{(t)}\|_{C^{2,\alpha}(\overline{E_1})}\le\|u\|_{C^{2,\alpha}(\overline{\Omega_\infty})}<\infty.
\]
Hence by the Arzel\'a–Ascoli theorem and the existence of $u^\sharp$,  $u^{(t)}$ converges to the function $u^\sharp$ in $C^2(\overline{E_1})$ as $t\rightarrow+\infty$. Since $u^{(t)}$ solves $$-\Delta u^{(t)}=f(u^{(t)})\quad \text{in }E_1,$$ letting $t\rightarrow+\infty$ implies that $u^\sharp$ is a positive solution to \eqref{1d solution}. It then follows from Lemma \ref{lemma of existence and uniqueness of 1d solution} that $u^\sharp \equiv\Bar{u}$. This completes the proof of the lemma.
\end{proof}

\begin{lemma}[Uniqueness]\label{lemma of uniqueness}
There exists a unique bounded solution $u\in C^2(\overline{\Omega_\infty})$ to \eqref{slee in a strip}. Moreover, this unique solution is odd in $x_1$.
\end{lemma}
\begin{proof}
The existence has been proved in Lemma \ref{lemma of existence}. 

Let $u,v$ be two bounded solutions to \eqref{slee in a strip}. Then $\min\{u|_{\overline{E_\infty}},v|_{\overline{E_\infty}}\}$ is a positive supersolution to \eqref{stream function of type3} in $E_\infty$. Recall the subsolution $\ubar{u}$ defined by \eqref{definition of subsolution}. In view of \eqref{asymptotics of gradient u}, one can choose $h$ so large that $\ubar{u}\le \min\{u|_{\overline{E_\infty}},v|_{\overline{E_\infty}}\}$ in $E_\infty$. Following the proof of Lemma \ref{lemma of existence}, there exists a bounded solution $w$ to \eqref{slee in a strip} such that 
\[
w(x)\le\min\{u(x),v(x)\}\quad\text{for all }x\in E_\infty.
\]
Now for every $k>0$, we have
\begin{align*}
\int_{-1}^{1}(w\partial_{x_1}u-u\partial_{x_1}w)(k,x_2)\,dx_2=\int_{E_k}(w\Delta u-u\Delta w)\,dx=\int_{E_k}uw\left(\frac{f(w)}{w}-\frac{f(u)}{u}\right)\,dx.
\end{align*}
Since $0<w\le u$ in $E_k$ and $\frac{f(s)}{s}$ is strictly decreasing, it holds that $$uw\left(\frac{f(w)}{w}-\frac{f(u)}{u}\right)\ge0.$$ Letting $k\rightarrow+\infty$ and recalling \eqref{asymptotics of gradient u} give
\begin{align*}
0=\int_{E_\infty}uw\left(\frac{f(u)}{u}-\frac{f(w)}{w}\right)\,dx.
\end{align*}
This implies $u=w$ in $E_\infty$. Similarly, $v=w$ in $E_\infty$, and thus, $u=v$ in $E_\infty$.

Now the functions $\tilde{u}(x)=-u(-x_1,x_2)$  and $\tilde{v}=-v(-x_1,x_2)$  are also bounded solutions to \eqref{slee in a strip}. By the previous step, we have 
\[
u(x)=v(x)=\tilde{u}(x)=-u(-x_1,x_2)= \tilde{v}(x)=-v(-x_1,x_2)\quad \text{for all }x\in E_\infty.
\]
Hence the lemma immediately follows.
\end{proof}

Incorporating \eqref{monotonicity in half strip}, \eqref{monotonicity in a quarter strip}, \eqref{symmetry in half strip}, \eqref{boundary velocity2}, Lemma \ref{lemma of asymptotics} and Lemma \ref{lemma of uniqueness}, we arrive at the following proposition.

\begin{pro}
There exists a unique bounded solution $u\in C^{2}(\overline{\Omega_\infty})$ to \eqref{slee in a strip}. Moreover, this unique solution satisfies that

{\rm (\romannumeral1)} $u(x_1,x_2)=-u(-x_1,x_2)=u(x_1,-x_2)$ for all $x\in\Omega_\infty$;

{\rm (\romannumeral2)} $\partial_{x_1}u>0$ in $\Omega_\infty$, and $\partial_{x_2}u>0$ in $(0,\infty)\times[-1,0)$;

{\rm (\romannumeral3)}
$\lim_{x_1\rightarrow+\infty}|\nabla u(x_1,x_2)-(0,\Bar{u}'(x_2))|=0$ uniformly in $x_2\in[-1,1]$.
\end{pro}

Finally, assuming $f$ is smooth and introducing $\bv=\nabla^\perp u$, we prove the following proposition.
\begin{pro}
There exists a bounded solution $\bv\in C^\infty(\overline{\Omega_\infty})$ to the Euler system in $\Omega_\infty$ satisfying $\Theta(\bv, \overline{\Omega_\infty})=S^1_+$ and the following properties. 
    \begin{enumerate}
\item[(a)] $v_1(-x_1,x_2)=v_1(x_1,-x_2)=-v_1(x_1,x_2)$ for all $x\in\overline{\Omega_\infty}$, and $v_1<0$ in $(0,\infty)\times[-1,0)$;

\item[(b)] $v_2(-x_1,x_2)=v_2(x_1,-x_2)=v_2(x_1,x_2)$ for all $x\in\overline{\Omega_\infty}$, and $v_2>0$ in $\Omega_\infty$;

\item[(c)] $\lim_{x_1\rightarrow+\infty}|\bv(x_1,x_2)-(-\Bar{u}'(x_2),0)|=0$ uniformly in $x_2\in[-1,1]$.
\end{enumerate}
\end{pro}

\begin{figure}[h]
\centering
\includegraphics[height=6cm, width=14cm]{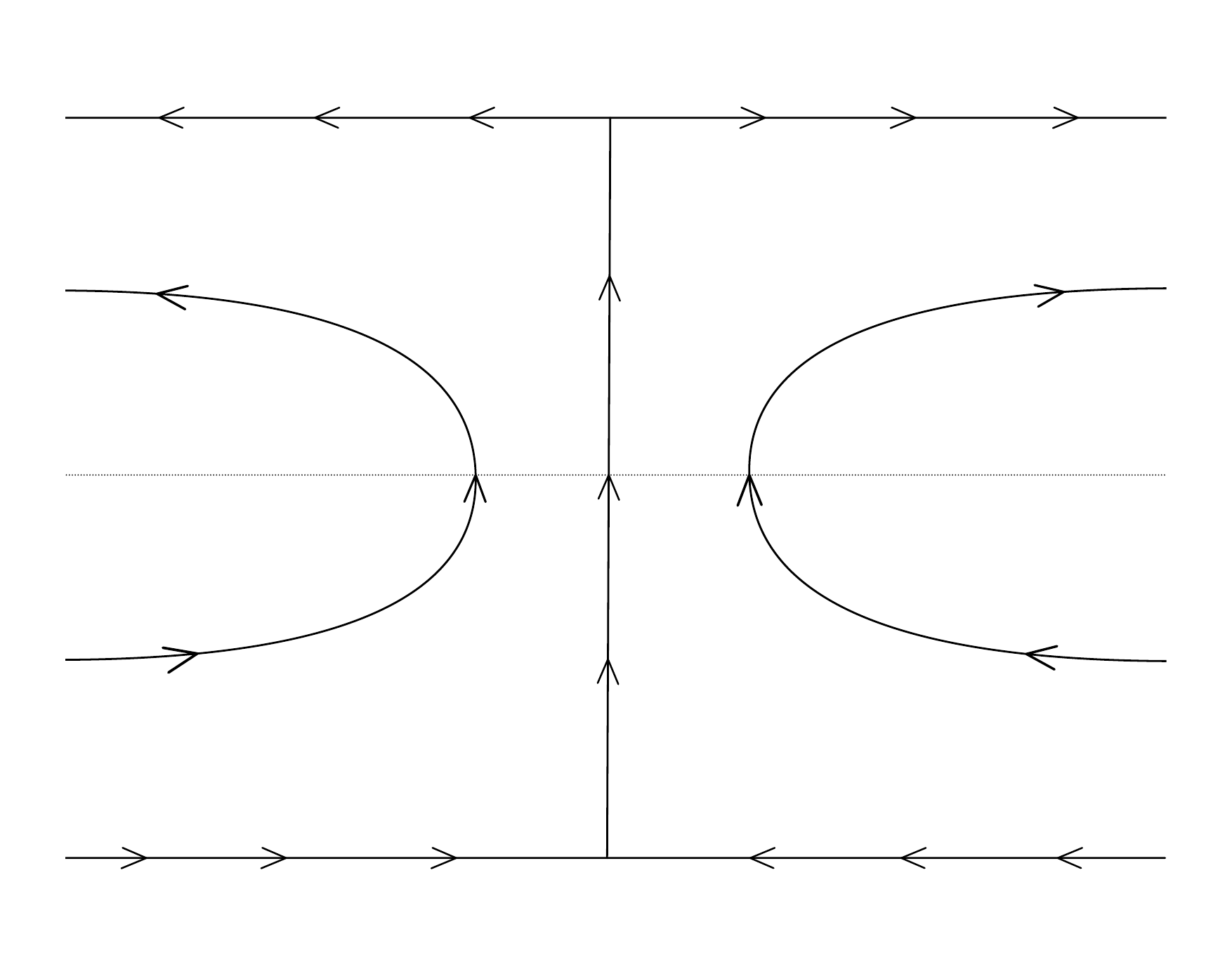}
\caption{Type {\thr} flow in $\Omega_\infty$}
\label{figure of type three flow in strip}
\end{figure}

This Type {\thr} flow exhibits streamlines shown in Figure \ref{figure of type three flow in strip}. There are exactly two stagnation points located at $(0,1)$ and $(0,-1)$.


\medskip 

{\bf Acknowledgement.}
The research of Gui is supported by University of Macau research grants  CPG2024-00016-FST, SRG2023-00011-FST, UMDF Professorial Fellowship of Mathematics, and Macao SAR FDCT 0003/2023/RIA1. 
 The research of  Xie is partially supported by  NSFC grants 12250710674 and 11971307, the Fundamental Research Funds for the Central Universities, Natural Science Foundation of Shanghai 21ZR1433300,  and Program of Shanghai Academic Research Leader 22XD1421400.

\bigskip

\bibliographystyle{plain}

\begin{thebibliography}{10}

\bibitem{ACM}
F. Alessio, A. Calamai, and P. Montecchiari,
{Saddle-type solutions for a class of semilinear elliptic equations}, {\it Adv. Differential Equations}, {\bf 12} (2007), no.4, 361--380.

\bibitem{AC1}
L. Ambrosio and X. Cabr\'e,  {Entire solutions of semilinear elliptic equations in $\mbR^3$ and a conjecture of De Giorgi}, {\it J. Amer. Math. Soc.}, {\bf 13} (2000), no. 4, 725--739.



\bibitem{Bat}
G. K. Batchelor,
{\it An introduction to fluid dynamics}, Second paperback edition, Cambridge Math. Lib. Cambridge University Press, Cambridge, 1999. 

\bibitem{BCN96}
H. Berestycki, L. Caffarelli, and L. Nirenberg, Inequalities for second-order elliptic equations with applications to unbounded domains I, {\it Duke Math. J.}, {\bf 81} (1996), no.2, 467--494.

\bibitem{BCN}
H. Berestycki, L. Caffarelli, and L. Nirenberg,  {Further qualitative properties for elliptic equations in unbounded domains}, {\it Ann. Scuola Norm. Sup. Pisa Cl. Sci.}, {\bf 25} (1997), 69--94.

\bibitem{BHR}
H. Berestycki, F. Hamel, and L. Rossi,
{Liouville-type results for semilinear elliptic equations in unbounded domains}, {\it Ann. Mat. Pura Appl.}, {\bf 4} (2007), no.3, 469--507.

\bibitem{BL}
H. Berestycki and P. L. Lions,
{Some applications of the method of super and subsolutions}, {\it Bifurcation and nonlinear eigenvalue problems (Proc., Session, Univ. Paris XIII, Villetaneuse, 1978)}, pp. 16–41, Lecture Notes in Math., {\bf 782} Springer, Berlin, 1980.

\bibitem{BN}
H. Berestycki and L. Nirenberg,  {On the method of moving planes and the sliding method}, {\it Bol. Soc. Brasil. Mat. (N.S.)}, {\bf 22} (1991), no.1, 1--37.

\bibitem{CL}
\'A. Castro and D. Lear,  {Traveling waves near Couette flow for the 2D Euler
equation}, {\it Comm. Math. Phys.}, {\bf 400} (2023), no.3, 2005--2079.

\bibitem{CLi}
O. Chodosh and C. Li, {Stable minimal hypersurfaces in $\mbR^4$}, to appear in Acta Mathematica.

\bibitem{CDG}
P. Constantin, T. D. Drivas, D. Ginsberg,  {Flexibility and rigidity in steady fluid motion}, {\it Comm. Math. Phys.}, {\bf 385} (2021), no.1, 521--563.

\bibitem{CEW}
M. Coti Zelati, T. Elgindi, and K. Widmayer, {Stationary structures near the Kolmogorov and Poiseuille flows in the 2d Euler equations}, {\it Arch. Ration. Mech. Anal.}, {\bf 247} (2023), no.1, Paper No. 12, 37 pp.

\bibitem{DFP}
H. Dang, P. C. Fife, and L. A. Peletier, {Saddle solutions of the bistable diffusion equation}, {\it Z. Angew. Math. Phys.}, {\bf 43} (1992), no.6, 984--998.

\bibitem{DKP}
M. del Pino,, M. Kowalczyk, and F. Pacard, 
{Moduli space theory for the Allen-Cahn equation in the plane}, {\it Trans. Amer. Math. Soc.}, {\bf 365} (2013), no.2, 721--766.

\bibitem{DKPW}
M. del Pino, M. Kowalczyk, F. Pacard, and J. Wei,
{Multiple-end solutions to the Allen-Cahn equation in $\mbR^2$}, {\it J. Funct. Anal.}, {\bf 258} (2010), no.2, 458--503.

\bibitem{Denisov}
S. A. Denisov, Infinite superlinear growth of the gradient for the two-dimensional Euler equation, {\it Discrete Contin. Dyn. Syst.}, {\bf 23} (2009), no. 3, 755--764.

\bibitem{DR}
P. G. Drazin and W. H.  Reid,  Hydrodynamic stability, Second edition,  Cambridge University Press, Cambridge, 2004.

\bibitem{ElgindiHuang}
T. M. Elgindi and Y. Huang,  Regular and Singular Steady States of 2D incompressible Euler equations near the Bahouri-Chemin Patch,  arXiv:2207.12640.

\bibitem{Evans}
L. C. Evans. {\it Partial differential equations}, Second edition, American Mathematical Society, 2010.

\bibitem{EG}
L. C. Evans, R. F. Gariepy,
{\it Measure theory and fine properties of functions}, Stud. Adv. Math.
CRC Press, Boca Raton, FL, 1992. 



\bibitem{FMM}
L. Franzoi, N. Masmoudi, and R. Montalto, {Space quasi-periodic steady Euler flows close to the inviscid Couette flow}, 	arXiv: 2303.03302.

\bibitem{GG}
N. Ghoussoub and C. Gui,  {On a conjecture of De Giorgi and some related problems}, {\it Math. Ann.}, {\bf 311} (1998), no. 3, 481--491.

\bibitem{GNN}
B. Gidas, W. M. Ni, and L. Nirenberg,  {Symmetry and related properties via the maximum principle}, {\it Comm. Math. Phys.}, {\bf 68} (1979), no.3, 209--243.

\bibitem{GT}
D. Gilbarg and N. Trudinger. {\it Elliptic partial differential equations of second order}, 2nd Ed., Springer-Verlag, Berlin,  1983.

\bibitem{GPSY}
J. G\'omez-Serrano, J. Park, J. Shi, and Y. Yao, {Symmetry in stationary and uniformly rotating solutions of active scalar equations}, {\it Duke Math. J.}, {\bf 170} (2021), no.13, 2957--3038.

\bibitem{Gui0}
C. Gui, {Hamiltonian identities for elliptic partial differential equations}, {\it J. Funct. Anal.}, {\bf 254} (2008), no.4, 904--933.

\bibitem{Gui}
C. Gui, {Symmetry of some entire solutions to the Allen-Cahn equation in two dimensions}, {\it J. Differential Equations}, {\bf 252} (2012), no.11, 5853--5874.

\bibitem{GLW}
C. Gui, Y. Liu, and J. Wei,
{On variational characterization of four-end solutions of the Allen-Cahn equation in the plane}, {\it J. Funct. Anal.}, {\bf 271} (2016), no.10, 2673--2700.


\bibitem{HN}
F. Hamel and N. Nadirashvili,  { Shear flows of an ideal fluid and elliptic equations in unbounded domains}, {\it Comm. Pure Appl. Math.}, {\bf 70} (2017),  590--608.

\bibitem{HN2}
F. Hamel and N. Nadirashvili, {A Liouville theorem for the Euler equations in the plane}, {\it Arch. Ration. Mech. Anal.}, {\bf 233} (2019), 599--642.

\bibitem{HN3}
F. Hamel and N. Nadirashvili, {Circular flows for the Euler equations in two-dimensional annular domains, and related free boundary problems}, {\it J. Eur. Math. Soc.}, {\bf 25} (2023), no.1, 323--368.

\bibitem{HouLuo}
G. Luo and T. Y. Hou, Formation of finite-time singularities in the 3D axisymmetric Euler equations: a numerics guided study,
{\it SIAM Rev.},  {\bf 61} (2019), no. 4, 793--835.

\bibitem{KS}
A. Kiselev and V. Sverak,
Small scale creation for solutions of the incompressible two-dimensional Euler equation, {\it Ann. of Math.}, {\bf 180} (2014), no. 3, 1205--1220.

\bibitem{KLP1}
M. Kowalczyk, Y. Liu, and F. Pacard, 
{The space of 4-ended solutions to the Allen-Cahn equation in the plane}, {\it Ann. Inst. H. Poincar\'e C Anal. Non Lin\'eaire}, {\bf 29} (2012), no.5, 761--781.

\bibitem{KLP2}
M. Kowalczyk, Y. Liu, and F. Pacard, 
{Towards classification of multiple-end solutions to the Allen-Cahn equation in $\mbR^2$}, {\it Netw. Heterog. Media}, {\bf 7} (2012), no.4, 837--855.

\bibitem{KLP3}
M. Kowalczyk, Y. Liu, and F. Pacard, 
{The classification of four-end solutions to the Allen-Cahn equation on the plane}, {\it Anal. PDE}, {\bf 6} (2013), no.7, 1675--1718.

\bibitem{KLPW}
M. Kowalczyk, Y. Liu, F. Pacard, and J. Wei,
{End-to-end construction for the Allen-Cahn equation in the plane}, {\it Calc. Var. Partial Differential Equations}, {\bf 52} (2015), no.1-2, 281--302.


\bibitem{LLSX}
C. Li, Y. L\"{u}, H. Shahgholian, and C. Xie, {Steady solutions for the Euler system in an infinitely long nozzle}, arXiv: 2203.08375.

\bibitem{LZ}
Z. Lin and C. Zeng, \newblock Inviscid dynamical structures near Couette flow, \newblock {\em Arch. Ration. Mech. Anal.}, {\bf 200} (2011), no.3, 1075--1097.

\bibitem{N}
N. Nadirashvili, \newblock On stationary solutions of two-dimensional Euler equation, \newblock {\em Arch. Ration. Mech. Anal.}, {\bf 209} (2013), no.3, 729–745.

\bibitem{R}
D. Ruiz, \newblock Symmetry results for compactly supported steady solutions of the 2D Euler equations, \newblock {\em Arch. Ration. Mech. Anal.}, {\bf 247} (2023), no.3, Paper No. 40, 25 pp.

\bibitem{Schoen}
R. Schoen, Estimates for stable minimal surfaces in three-dimensional manifolds. {\em Seminar on minimal submanifolds}, 111–126. Ann. of Math. Stud., 103
Princeton University Press, Princeton, NJ, 1983

\bibitem{SSY}
R. Schoen, L. Simon, and S. T. Yau, \newblock Curvature estimates for minimal hypersurfaces, \newblock {\em Acta Math.}, {\bf 134} (1975), no.3-4, 275--288.

\bibitem{Simons}
J. Simons, Minimal varieties in riemannian manifolds, {\em Ann. of Math.}, {\bf 88} (1968), 62--105.


\bibitem{WW}
K. Wang and J. Wei, {Finite Morse index implies finite ends}, {\em Comm. Pure Appl. Math.}, {\bf 72} (2019), no.5, 1044--1119.

\bibitem{WZ}
Y. Wang and W. Zhan, {On the rigidity of the 2D incompressible Euler equations}, arXiv: 2307.00197
\end{thebibliography}

\end{document}